\documentclass[a4paper]{article}

\usepackage{amssymb,amsmath, amsfonts}
\usepackage{authblk}
\usepackage{pst-grad} 

\usepackage{tikz-cd}
\usetikzlibrary{matrix,arrows,decorations.pathmorphing}

\usepackage{cmdtrack}
\usepackage{amsthm}
\usepackage{url}

\newcommand\otm{\otimes}
\newcommand\str{\rightarrow}
\newcommand\mj{\mbox{\bf 1}}
\newcommand\id{\mbox{\bf 1}}
\newcommand\set{\mbox{\emph{Set}}}

\newcommand\md{\mu^\lozenge}
\newcommand\ed{\eta^\lozenge}
\newcommand\ms{\mu^\square}
\newcommand\es{\eta^\square}

\newcommand\mR{\mbox{$\mathbf{R}$}}

\newcommand\mZ{\mbox{$\mathbf{Z}$}}
\newcommand\Int{\mbox{\rm Int}}

 \newtheorem{thm}{Theorem}[section]
 \newtheorem{prop}[thm]{Proposition}
 \newtheorem{lem}[thm]{Lemma}
 \newtheorem{cor}[thm]{Corollary}

 \newtheorem{rem}[thm]{Remark}
 
\numberwithin{equation}{section}

\newcommand{\HDS}{\vrule width0pt height2.3ex depth1.05ex\displaystyle}
\newcommand{\f}[2]{{\frac{\HDS #1}{\HDS #2}}}

\begin{document}

\title{Spheres as Frobenius objects}
\author{Djordje Barali\' c, Zoran Petri\'c and Sonja Telebakovi\' c}
\affil{Mathematical Institute SANU,\\ Knez Mihailova 36, p.f.\
367,\\ 11001 Belgrade, Serbia

\vspace{1ex}

Faculty of Mathematics,
\\Studentski trg 16,\\ 11000 Belgrade, Serbia

\vspace{1ex}

\texttt{djbaralic@mi.sanu.ac.rs}, \texttt{zpetric@mi.sanu.ac.rs},
\texttt{sonjat@matf.bg.ac.rs }}

\date{}
\maketitle

\vspace{-3ex}

\begin{abstract}

Following the pattern of the Frobenius structure usually assigned
to the 1-dimensional sphere, we investigate the Frobenius
structures of spheres in all other dimensions. Starting from
dimension $d=1$, all the spheres are commutative Frobenius objects
in categories whose arrows are ${(d+1)}$-dimensional cobordisms.
With respect to the language of Frobenius objects, there is no
distinction between these spheres---they are all free of
additional equations formulated in this language. The
corresponding structure makes out of the 0-dimensional sphere not
a commutative but a symmetric Frobenius object. This sphere is
mapped to a matrix Frobenius algebra by a 1-dimensional
topological quantum field theory, which corresponds to the
representation of a class of diagrammatic algebras given by
Richard Brauer.

\end{abstract}

\vspace{.3cm}


\vspace{.5ex}

\noindent {\small {\it Keywords$\,$}: symmetric monoidal category,
commutative Frobenius object, oriented manifold, cobordism, normal
form, coherence, topological quantum field theory, Brauerian
representation}

\vspace{.5ex}

\noindent {\small {\it Acknowledgements$\,$}: We are grateful very
much to Joachim Kock for some useful comments concerning the
previous version of this paper and for pointing out to us the
reference \cite{S95}. We would like to thank the anonymous referee
for suggestions which helped to improve the paper. This work was
supported by projects 174020, 174026 and 174032 of the Ministry of
Education, Science, and Technological Development of the Republic
of Serbia. }

\section{Introduction}

A Frobenius structure of one dimensional sphere $S^1$ is
thoroughly investigated in a series of papers and books (see
\cite{D89}, \cite{A96}, \cite{K03} and references therein). It is
not the case that $S^1$ as a commutative Frobenius object of the
category of 2-cobordisms is dealt with separately, but always in
the context of two dimensional topological quantum field theories
and in connection with Frobenius algebras. A Frobenius structure
of spheres of other dimensions is investigated in \cite{D89}
and~\cite{S95}.

It is straightforward to conclude that for every $d\geq 1$, the
sphere $S^{d-1}$ is a symmetric Frobenius object in the category
$dCob$ of $d$-cobordisms. Also, it is straightforward to conclude
that for every $d\geq 2$, the sphere $S^{d-1}$ is a commutative
Frobenius object in this category. (The author of \cite{S95}
claims in Proposition~1 that every sphere is a commutative
Frobenius object, which is not true for the case of $S^0$.) This
means that increasing the dimension of a sphere from 0 to 1
produces a narrowing of the class of symmetric to the class of
commutative Frobenius objects. Hence, it is natural to ask the
following question: how many such steps are there, which produce
new classes of Frobenius objects, induced by increasing the
dimension of spheres?

The notion of commutative Frobenius object is not Post complete,
i.e.\ adding a new equality between the canonical arrows (those
relevant for the Frobenius structure) does not produce a
collapse---some canonical arrows with the same source and target
remain different. Hence, there are different classes of
commutative Frobenius objects. If for a pair of different closed
2-manifolds, one forms the corresponding equality of canonical
arrows, then all the commutative Frobenius objects satisfying this
new equality form a proper subclass of commutative Frobenius
objects. There are infinitely many such classes and
\cite[Proposition~2.4]{PT17} provides a way to classify all the
commutative Frobenius objects into classes corresponding to pairs
of closed 2-manifolds.

For example, the class of commutative Frobenius objects satisfying
the equality: comultiplication followed by multiplication equals
identity (a \emph{special} Frobenius algebra is such an object) is
a proper subclass of the commutative Frobenius objects satisfying
the equality: unit followed by comultiplication followed by
multiplication followed by counit equals unit followed by counit.
In terms of 2-manifolds, the latter class corresponds to the pair
consisting of the torus $S^1\times S^1$ and the sphere $S^2$.

The purpose of this paper is to show that no proper subclass of
commutative Frobenius objects includes $S^{d-1}$, for $d\geq 2$.
In order to do this, we construct a symmetric monoidal category
$K$ with a universal commutative Frobenius object, and show that
for every $d\geq 2$, every symmetric monoidal functor from $K$ to
$dCob$ that maps this object to $S^{d-1}$ is faithful.

The paper is organized so that some basic notions from category
theory, which are necessary for understanding the results, are
given in this introductory section. The category $dCobS$, whose
objects are finite collections of $(d\!-\!1)$-dimensional spheres
and arrows are equivalence classes of topological
$d$-co\-bor\-d\-isms, is introduced in Sections \ref{dCobS} and
\ref{orientation}. This category is an ambient for a Frobenius
object $S^{d-1}$. The category $dCobS$ is a full subcategory of
the category $dCob$ whose objects are the $(d\!-\!1)$-dimensional
closed topological manifolds.

In Section~\ref{why}, we justify our restriction of objects of the
category of $d$-cobordisms to collections of spheres. The results
of this section heavily depend on some topological facts that are
listed in Section~\ref{topological}. In Section~\ref{Frobenius},
the pattern followed by us is explained in order to define a
Frobenius structure of a sphere.

Section~\ref{Brauerian} is devoted to the case of $S^0$ and a
classical result of Richard Brauer concerning a matrix
representation of a class of diagrammatic algebras. This matrix
representation is generalized by Do\v sen and the second author
(see \cite{DP06} and \cite{DP12}) to cover a category and not just
a monoid of diagrams. This generalization is a one dimensional
topological quantum field theory that maps $S^0$ to a matrix
Frobenius algebra, which is usually the first example of a
Frobenius algebra one finds in the literature.

Section \ref{categoryK} serves to define a symmetric strict
monoidal category $K$ with a universal commutative Frobenius
object in it. This category is built out  of a syntax material.
Technical details of this construction are given in
Section~\ref{appendixK}. A normal form for arrows of this category
is given in Section~\ref{normal form}.

The main result of Section \ref{faithfulness} is that, for every
$d\geq 2$, the category $K$ is embeddable into $dCobS$. The image
of the universal Frobenius object through this embedding is the
sphere $S^{d-1}$. Such a result is a completeness result from the
point of view of a logician and a coherence result from the point
of view of a category theorist. It says that with respect to the
language of Frobenius objects there is no distinction between
spheres starting from dimension $d=1$, i.e.\ they are all free of
additional equations formulated in this language. This provides
the answer to the question from the second paragraph.

Almost all the categories we deal with in this paper are
\emph{skeletal} in the sense that there are no two different
isomorphic objects in them. Hence, all the monoidal categories
mentioned below will be strict monoidal. In this way we lose some
interesting combinatorics tied to associativity, but this enables
us to emphasize the combinatorial structure we investigate.

A \emph{strict monoidal} category is a triple
$(\mathcal{M},\otimes,e)$ consisting of a category $\mathcal{M}$,
a bifunctor $\otimes:\mathcal{M}\times \mathcal{M}\str
\mathcal{M}$, which is associative, and an object $e$, which is a
left and right unit for $\otimes$. It is \emph{symmetric} when
there is a natural transformation $\tau$ with components
\[
\tau_{A,B}:A\otimes B\str B\otimes A,
\]
which means that for every pair of arrows $f:A\str A'$ and
$g:B\str B'$ the diagram

\begin{center}
\begin{picture}(120,50)(0,-5)

\put(0,40){\makebox(0,0){$A\otimes B$}}
\put(0,0){\makebox(0,0){$A'\otimes B'$}}
\put(80,40){\makebox(0,0){$B\otimes A$}}
\put(80,0){\makebox(0,0){$B'\otimes A'$}}

\put(40,47){\makebox(0,0){$\tau_{A,B}$}}
\put(40,-7){\makebox(0,0){$\tau_{A',B'}$}}
\put(-15,20){\makebox(0,0){$f\otimes g$}}
\put(95,20){\makebox(0,0){$g\otimes f$}}

\put(20,40){\vector(1,0){40}} \put(20,0){\vector(1,0){40}}
\put(0,30){\vector(0,-1){20}} \put(80,30){\vector(0,-1){20}}
\end{picture}
\end{center}
commutes, this transformation is self-inverse, i.e.\
$\tau_{B,A}\circ\tau_{A,B}=\mj_{A\otimes B}$, and it satisfies
$\tau_{A\otimes B,C}=(\tau_{A,C}\otimes
\mj_B)\circ(\mj_A\otimes\tau_{B,C})$ (cf.\ the equations
(\ref{st}), (\ref{ct}), (\ref{fn}), (\ref{nt}), (\ref{iv}) and
(\ref{hx}) of Section~\ref{appendixK}). The main example of
symmetric strict monoidal categories in this paper are the
categories $dCobS$ and $dCob$ introduced in Section~\ref{dCobS}.

A \emph{monoid} $(M,\md:M\otimes M\str M,\ed:e\str M)$ in a strict
monoidal category $\mathcal{M}$ is a triple consisting of an
object $M$ of $\mathcal{M}$, and two arrows $\md$ and $\ed$ of
$\mathcal{M}$, such that the following diagrams commute

\begin{center}
\begin{picture}(250,60)(0,-5)

\put(0,40){\makebox(0,0){$M\otimes M\otimes M$}}
\put(0,0){\makebox(0,0){$M\otimes M$}}
\put(100,40){\makebox(0,0){$M\otimes M$}}
\put(100,0){\makebox(0,0){$M$}}

\put(53,47){\makebox(0,0){$\md\otimes\mj_M$}}
\put(50,-7){\makebox(0,0){$\md$}}
\put(-20,20){\makebox(0,0){$\mj_M\otimes\md$}}
\put(110,20){\makebox(0,0){$\md$}}

\put(30,40){\vector(1,0){50}} \put(20,0){\vector(1,0){60}}
\put(0,30){\vector(0,-1){20}} \put(100,30){\vector(0,-1){20}}

\put(150,40){\makebox(0,0){$M\otimes M$}}
\put(270,40){\makebox(0,0){$M\otimes M$}}
\put(210,40){\makebox(0,0){$M$}} \put(210,0){\makebox(0,0){$M$}}

\put(200,40){\vector(-1,0){30}} \put(220,40){\vector(1,0){30}}
\put(210,30){\vector(0,-1){20}} \put(160,30){\vector(2,-1){40}}
\put(260,30){\vector(-2,-1){40}}

\put(188,47){\makebox(0,0){$\ed\!\otimes\mj_M$}}
\put(235,47){\makebox(0,0){$\mj_M\otimes\! \ed$}}
\put(218,24){\makebox(0,0){$\mj_M$}}
\put(170,17){\makebox(0,0){$\md$}}
\put(250,17){\makebox(0,0){$\md$}}

\end{picture}
\end{center}
(cf.\ the equations (\ref{as}) and (\ref{un}) of
Section~\ref{appendixK}).

A \emph{comonoid} $(M,\ms:M\str M\otimes M,\es:M\str e)$ in
$\mathcal{M}$ is defined in a dual manner (cf.\ the equations
(\ref{ca}) and (\ref{cu}) of Section~\ref{appendixK}). A
\emph{Frobenius object} in $\mathcal{M}$ is a quintuple
\[
(M,\md:M\otimes M\str M,\ed:e\str M,\ms:M\str M\otimes M,\es:M\str
e)
\]
such that $(M,\md,\ed)$ is a monoid, $(M,\ms,\es)$ is a comonoid,
and the following \emph{Frobenius equations} (cf.\ the equations
(\ref{fb}) of Section~\ref{appendixK}) hold
\[
(\mj_M\otimes\md)\circ(\ms\otimes\mj_M)=\ms\circ\md=(\md\otimes\mj_M)\circ(\mj_M\otimes\ms).
\]

If $\mathcal{M}$ is symmetric, then a Frobenius object
$(M,\md,\ed,\ms,\es)$ is \emph{commutative} when
\[
\md\circ\tau_{M,M}=\md\quad \mbox{\rm and}\quad
\tau_{M,M}\circ\ms=\ms
\]
(cf.\ the equations (\ref{cm}) and (\ref{cocm}) of
Section~\ref{appendixK}, which are interderivable in the presence
of other equations), and it is \emph{symmetric} when
\[
\es\circ\md\circ\tau_{M,M}=\es\circ\md\quad \mbox{\rm and}\quad
\tau_{M,M}\circ\ms\circ\ed=\ms\circ\ed.
\]

A functor between two symmetric strict monoidal categories is
\emph{symmetric monoidal} when it preserves the symmetric monoidal
structure on the nose, i.e.\ it maps tensor to tensor, unit to
unit and symmetry to symmetry. According to our intention to work
with strict monoidal structures, by a $d$-dimensional
\emph{topological quantum field theory} ($d$TQFT) we mean here a
symmetric monoidal functor between the category $dCob$ and a
strict monoidal category equivalent to the category of finite
dimensional vector spaces over a chosen field. This
strictification is supported by \cite[Section~XI.3,
Theorem~1]{ML71}.

In some parts of the text, a natural number (finite ordinal) $n$
is considered to be the set $\{0,\ldots,n-1\}$. It will be clear
from the context when this is assumed.

\section{The category $dCobS$}\label{dCobS}

By a $d$-\emph{manifold} we mean here a compact, oriented
$d$-dimensional $\partial$-manifold (see
Section~\ref{orientation}). It is \emph{closed} when its boundary
is empty.

For $d\geq 1$ and $i\in \mathbf{N}$, let $S_i$ be the
$(d\!-\!1)$-dimensional sphere in $\mathbf{R}^d$ with the center
$(3i,0,\ldots,0)$ and the radius~1. Assume that an orientation of
$S_0$ is chosen, and that $S_i$ is oriented so that the
translation by the vector $(3i,0,\ldots,0)$ is an orientation
preserving homeomorphism from $S_0$ to $S_i$. Let $\underline{0}$
denote the empty set, and for $n>0$, let $\underline{n}$ denote
the closed $(d\!-\!1)$-manifold $S_0\cup\ldots\cup S_{n-1}$.

Let $M$ be a $d$-manifold such that its boundary $\partial M$ is a
disjoint union of $\Sigma_0$ homeomorphic to $\underline{n}$ and
$\Sigma_1$ homeomorphic to $\underline{m}$. We assume that the
orientations of  $\Sigma_0$ and $\Sigma_1$ are induced from the
orientation of $M$ (see Section~\ref{orientation}).

Let $f_0:\underline{n}\str M$ and $f_1:\underline{m}\str M$ be two
embeddings whose images are respectively $\Sigma_0$ and
$\Sigma_1$. Assume that $f_0$ preserves, while $f_1$ reverses the
orientation. The triple $(M,f_0,f_1)$ is a $d$-\emph{cobordism},
or simply a cobordism, from $\underline{n}$ to $\underline{m}$. We
call $\Sigma_0$ and $\Sigma_1$, respectively, the \emph{ingoing}
and \emph{outgoing} boundary of $M$ in this cobordism.

Two $d$-cobordisms $K=(M,f_0,f_1)$ and $K'=(M',f'_0,f'_1)$ are
\emph{equivalent}, which we denote by $K\sim K'$, when there is an
orientation preserving homeomorphism $F:M\str M'$ such that the
following diagram commutes.

\begin{center}
\begin{picture}(120,60)

\put(0,30){\makebox(0,0){$\underline{n}$}}
\put(60,55){\makebox(0,0){$M$}} \put(60,5){\makebox(0,0){$M'$}}
\put(120,30){\makebox(0,0){$\underline{m}$}}
\put(25,50){\makebox(0,0){$f_0$}}
\put(95,50){\makebox(0,0){$f_1$}}
\put(25,10){\makebox(0,0){$f'_0$}}
\put(95,10){\makebox(0,0){$f'_1$}} \put(67,30){\makebox(0,0){$F$}}

\put(10,35){\vector(2,1){40}} \put(10,25){\vector(2,-1){40}}
\put(110,35){\vector(-2,1){40}} \put(110,25){\vector(-2,-1){40}}
\put(60,45){\vector(0,-1){30}}
\end{picture}
\end{center}

The category $dCobS$ has
$\underline{0},\underline{1},\underline{2},\ldots$ as objects and
the equivalence classes of $d$-cobordisms as arrows. The
\emph{identity arrow} from $\underline{n}$ to $\underline{n}$ in
$dCobS$ is the equivalence class of the $d$-cobordism
\begin{center}
\begin{picture}(120,20)

\put(0,5){\makebox(0,0){$\underline{n}$}}
\put(120,5){\makebox(0,0){$\underline{n}$}}
\put(60,5){\makebox(0,0){$\underline{n}\times I$}}

\put(25,15){\makebox(0,0){$\langle\id,c_0\rangle$}}
\put(95,15){\makebox(0,0){$\langle\id,c_1\rangle$}}

\put(10,5){\vector(1,0){35}} \put(110,5){\vector(-1,0){35}}

\end{picture}
\end{center}
where $I$ is the unit interval $[0,1]$, $\id$ is the identity map
on $\underline{n}$, $c_0,c_1:\underline{n}\str I$ are the constant
maps $c_0(x)=0$ and $c_1(x)=1$, and for $f:C\str A$, and $g:C\str
B$, the pairing $\langle f,g\rangle:C\str A\times B$ is defined by
$\langle f,g\rangle(c)=(f(c),g(c))$.

\emph{Composition} of cobordisms $(M,f_0,f_1):\underline{n}\str
\underline{m}$ and $(N,g_0,g_1):\underline{m}\str \underline{k}$
consists of the $d$-manifold $N+_{g_0,f_1}M$ obtained by gluing
(see Section \ref{orientation}) and two maps $j\circ f_0$ and
$i\circ g_1$, where $i:N\str N+_{g_0,f_1}M$ and $j:M\str
N+_{g_0,f_1}M$ are the embeddings in the corresponding pushout
diagram (see Section \ref{orientation}). Equivalence of cobordisms
is a congruence with respect to the composition.

When $d=2$, the category $dCobS$ is isomorphic to the category
\textbf{2-Cobord} of \cite[Section~4]{A96}. The category $dCobS$
is strict monoidal with respect to the sum on objects
($\underline{n}+\underline{m}=\underline{n+m}$) and the following
operation of ``putting side by side'' on arrows. First, for two
$d$-manifolds $N$ and $M$, we denote by $N+M$ the disjoint union
$(N\times\{0\})\cup (M\times\{1\})$, and for two functions
$f:\underline{n}\str N$ and $g:\underline{m}\str M$, we denote by
$f+g:\underline{n+m}\str N+M$ the following function
\[
(f+g)(x)=\left\{ \begin{array}{ll} (f(x),0), & x\in \underline{n} \\
(g(x-(3n,0,\ldots,0)),1), & x\not\in
\underline{n}.\end{array}\right.
\]
Then, the ``putting side by side'' of $(N,f_0,f_1)$ and
$(M,g_0,g_1)$ is the $d$-cobordism
\[
(N+M,f_0+g_0,f_1+g_1).
\]

The category $dCobS$ is also symmetric monoidal with respect to
the family of $d$-cobordisms $\underline{\tau_{n,m}}$, defined as
\begin{center}
\begin{picture}(200,20)

\put(0,5){\makebox(0,0){$\underline{n}+\underline{m}$}}
\put(200,5){\makebox(0,0){$\underline{m}+\underline{n}$,}}
\put(100,5){\makebox(0,0){$(\underline{n+m})\times I$}}

\put(35,15){\makebox(0,0){$\langle\id,c_0\rangle$}}
\put(165,15){\makebox(0,0){$\langle f,c_1\rangle$}}

\put(20,5){\vector(1,0){35}} \put(180,5){\vector(-1,0){35}}

\end{picture}
\end{center}
where $f:\underline{n+m}\str \underline{n+m}$ translates the
spheres $S_i$, $0\leq i\leq n-1$, by the vector $(3m,0,\ldots,0)$,
and the spheres $S_j$, $n\leq j\leq n+m-1$, by the vector
$(-3n,0,\ldots,0)$.

The category $dCobS$ is skeletal, i.e.\ there are no two different
isomorphic objects in $dCobS$. This is shown below (see
Section~\ref{Brauerian} and Corollary~\ref{skeletal}). It is a
full subcategory of the category $dCob$, whose objects are all
closed $(d\!-\!1)$-manifolds, and whose arrows are based on
arbitrary $d$-manifolds, and not only on those with boundaries
homeomorphic to collections of spheres. The symmetric monoidal
structure of the category $dCob$ is defined as for $dCobS$.

\section{Why spheres?}\label{why}

In this section we explain why we work in $dCobS$ and not in
$dCob$, and why we deal with topological and not with smooth
manifolds. The main reason is that dealing with arrows of $dCobS$
is simplified to a certain extent by ``irrelevance'' of gluing.
Section~\ref{topological} serves to prepare the ground for the
results of this section. The ambient consisting of collections of
spheres is sufficient for our purposes, since we investigate
spheres as Frobenius objects.

\begin{lem}\label{lems1}
If $f:\underline{1}\str \underline{1}$ is an orientation
preserving homeomorphism, then the cobordisms
$(\underline{1}\times
I,\langle\id,c_0\rangle,\langle\id,c_1\rangle)$ and
$(\underline{1}\times I,\langle\id,c_0\rangle,\langle
f,c_1\rangle)$ are equivalent.
\end{lem}

\begin{proof}
Let $F:\underline{1}\times I\str \underline{1}\times I$ be the
homeomorphism from Proposition~\ref{prop9} such that
$F(x,0)=(x,0)$ and $F(x,1)=(f(x),1)$. Then $F$ makes the following
diagram commutative.

\begin{center}
\begin{picture}(120,60)

\put(0,30){\makebox(0,0){$\underline{1}$}}
\put(65,55){\makebox(0,0){$\underline{1}\times I$}}
\put(65,5){\makebox(0,0){$\underline{1}\times I$}}
\put(130,30){\makebox(0,0){$\underline{1}$}}
\put(15,50){\makebox(0,0){$\langle\id,c_0\rangle$}}
\put(115,50){\makebox(0,0){$\langle\id,c_1\rangle$}}
\put(15,10){\makebox(0,0){$\langle\id,c_0\rangle$}}
\put(115,10){\makebox(0,0){$\langle f,c_1\rangle$}}
\put(72,30){\makebox(0,0){$F$}}

\put(10,35){\vector(2,1){40}} \put(10,25){\vector(2,-1){40}}
\put(120,35){\vector(-2,1){40}} \put(120,25){\vector(-2,-1){40}}
\put(65,45){\vector(0,-1){30}}
\end{picture}
\end{center}
\end{proof}

\begin{lem}\label{lems2}
If $u,v:\underline{1}\str \Sigma$ are two orientation preserving
homeomorphisms, then the cobordisms $K_1=(\Sigma\times I,\langle
v,c_0\rangle,\langle v,c_1\rangle)$, $K_2=(\Sigma\times I,\langle
v,c_0\rangle,\langle u,c_1\rangle)$ and $(\underline{1}\times
I,\langle\id,c_0\rangle,\langle\id,c_1\rangle)$ are equivalent.
\end{lem}

\begin{proof}
The homeomorphism $F$ in the center of the following diagram is
the one from Lemma~\ref{lems1} obtained for $f=v^{-1}\circ u$.
\begin{center}
\begin{picture}(120,130)

\put(0,60){\makebox(0,0){$\underline{1}$}}
\put(70,0){\makebox(0,0){$\Sigma\times I$}}
\put(70,40){\makebox(0,0){$\underline{1}\times I$}}
\put(70,120){\makebox(0,0){$\Sigma\times I$}}
\put(70,80){\makebox(0,0){$\underline{1}\times I$}}
\put(140,60){\makebox(0,0){$\underline{1}$}}

\put(15,0){\makebox(0,0){$\langle v,c_0\rangle$}}
\put(15,120){\makebox(0,0){$\langle v,c_0\rangle$}}
\put(125,0){\makebox(0,0){$\langle u,c_1\rangle$}}
\put(125,120){\makebox(0,0){$\langle v,c_1\rangle$}}

\put(25,40){\makebox(0,0){$\langle \id,c_0\rangle$}}
\put(25,80){\makebox(0,0){$\langle \id,c_0\rangle$}}
\put(115,40){\makebox(0,0){$\langle f,c_1\rangle$}}
\put(115,80){\makebox(0,0){$\langle \id,c_1\rangle$}}

\put(77,60){\makebox(0,0){$F$}}
\put(90,100){\makebox(0,0){${v^{-1}\times \id}$}}
\put(85,20){\makebox(0,0){${v\times \id}$}}

\put(70,30){\vector(0,-1){20}} \put(70,70){\vector(0,-1){20}}
\put(70,110){\vector(0,-1){20}}

\put(10,55){\vector(3,-1){40}} \put(10,65){\vector(3,1){40}}
\put(130,55){\vector(-3,-1){40}} \put(130,65){\vector(-3,1){40}}

\qbezier(5,50)(20,0)(50,0) \put(47,0){\vector(1,0){3}}
\qbezier(5,70)(20,120)(50,120) \put(47,120){\vector(1,0){3}}

\qbezier(135,50)(120,0)(90,0) \put(93,0){\vector(-1,0){3}}
\qbezier(135,70)(120,120)(90,120) \put(93,120){\vector(-1,0){3}}

\end{picture}
\end{center}
\end{proof}

\begin{lem}\label{lems3}
If $u,v:\underline{1}\str \Sigma$ are two orientation preserving
homeomorphisms, where $\Sigma$ is a part of the boundary of a
$d$-manifold $M$, then the cobordisms $(M,f+u+g,h)$ and
$(M,f+v+g,h)$ are equivalent.
\end{lem}

\begin{proof}
Let $K_1$ and $K_2$ be the cobordisms from Lemma~\ref{lems2}. For
$\underline{n}$ and $\underline{m}$ being the sources of $f$ and
$g$ respectively, we have
\begin{tabbing}
\hspace{1.5em}$(M,f+u+g,h)$ \= $\sim (M,f+u+g,h)\circ
\mj_{\underline{n+1+m}}$
\\[1ex]
\> $\sim (M,f+u+g,h)\circ
(\mj_{\underline{n}}+K_2+\mj_{\underline{m}})$
\\[1ex]
\> $= (M,f+v+g,h)\circ
(\mj_{\underline{n}}+K_1+\mj_{\underline{m}})$
\\[1ex]
\> $\sim (M,f+v+g,h)\circ \mj_{\underline{n+1+m}}$
\\[1ex]
\> $\sim (M,f+v+g,h)$.
\end{tabbing}
\end{proof}

By iterating Lemma~\ref{lems3} and an analogous result concerning
the outgoing boundary of $M$, we obtain the following result in
which ``connected components'' should be replaced by ``pairs of
points'', when $d=1$.

\begin{cor}\label{cors1}
Every arrow of $dCobS$ is completely determined by a $d$-manifold
and two sequences---one of connected components of the ingoing
boundary and the other of connected components of the outgoing
boundary.
\end{cor}

Hence, we may denote an arrow from $\underline{n}$ to
$\underline{m}$ by $(M,\Sigma_0,\Sigma_1)$, where
$\Sigma_0=(\Sigma_0^0,\ldots,\Sigma_0^{n-1})$ is a sequence of all
the connected components (or pairs of points, when $d=1$) of the
ingoing boundary and $\Sigma_1=(\Sigma_1^0,\ldots,\Sigma_1^{m-1})$
is a sequence of all the connected components of the outgoing
boundary of $M$.

\begin{prop}\label{props1}
Two cobordisms $(M,\Sigma_0,\Sigma_1)$ and $(N,\Delta_0,\Delta_1)$
are equivalent iff the corresponding sequences are of the same
length and there is a homeomorphism $F:M\str N$ such that for
every $i\in\{0,1\}$ and every $j$, the image of $F$ restricted to
$\Sigma^j_i$ is $\Delta^j_i$.
\end{prop}

\begin{proof}
The direction from left to right follows from the definition of
equivalence. For the other direction, for every $j$, let
$h_0^j:\underline{1}\str \Sigma_0^j$ be an orientation preserving
homeomorphism and let $h_1^j:\underline{1}\str \Sigma_1^j$ be an
orientation reversing homeomorphism. Define
$g_i^j:\underline{1}\str \Delta_i^j$ to be $F\circ h_i^j$. Then
$F$ underlies the equivalence of $(M,\sum_{j=0}^{n-1}
h_0^j,\sum_{j=0}^{m-1} h_1^j)$ and $(N,\sum_{j=0}^{n-1}
g_0^j,\sum_{j=0}^{m-1} g_1^j)$.
\end{proof}

However, if for $d\geq 3$ we allow closed $(d\!-\!1)$-manifolds
other than collections of spheres to be objects of the category of
$d$-cobordisms, then it would not be the case that the arrows of
such a category are determined just by manifolds and sequences of
ingoing and outgoing boundaries. For example, a solid torus with
the torus as the ingoing boundary and the empty set as the
outgoing boundary does not determine a 3-cobordism. The identity
map and an orientation preserving homeomorphism of the torus that
interchanges parallels and meridians define two different
3-cobordisms. By the result of Lickorish, \cite{L62}, every
closed, connected, 3-manifold is obtainable from $S^3$ by removing
a finite collection of solid tori, and then sewing them back. For
example, if one removes an unknotted solid torus from $S^3$ and
sew it back according to a homeomorphism of torus that
interchanges parallels and meridians, then the resulting
3-manifold is $S^2\times S^1$.

In case of the category of smooth $d$-cobordisms as arrows and
collections of spheres as objects, the analogues of Corollary
\ref{cors1} and Proposition \ref{props1} do not hold for every
$d$. For example, the manifold $S^{d-1}\times I$ with
$S^{d-1}\times \{0\}$ as the ingoing and $S^{d-1}\times \{1\}$ as
the outgoing boundary does not determine a $d$-cobordism. This is
shown as follows.

A \emph{pseudo-isotopy} of a smooth closed manifold $M$ is a
diffeomorphism $F$ of $M\times I$ that restricts to the identity
on $M\times\{0\}$. The restriction of $F$ to $M\times\{1\}$ is, up
to the identification of $M\times\{1\}$ with $M$, a diffeomorphism
$f:M\str M$. One says that $f$ is \emph{pseudo-isotopic} to the
identity.

By a definition analogous to the one given in Section \ref{dCobS}
(cf.\ \cite[1.2.17]{K03}), two smooth $d$-cobordisms
$(S^{d-1}\times I,\langle\id,c_0\rangle,\langle\id,c_1\rangle)$
and $(S^{d-1}\times I,\langle\id,c_0\rangle,\langle f,c_1\rangle)$
are equivalent when there is an orientation preserving
diffeomorphism $F:S^{d-1}\times I\str S^{d-1}\times I$ such that
the following diagram commutes.
\begin{center}
\begin{picture}(120,70)(0,-5)

\put(-2,30){\makebox(0,0){$S^{d-1}$}}
\put(75,55){\makebox(0,0){$S^{d-1}\times I$}}
\put(75,5){\makebox(0,0){$S^{d-1}\times I$}}
\put(153,30){\makebox(0,0){$S^{d-1}$}}
\put(15,50){\makebox(0,0){$\langle\id,c_0\rangle$}}
\put(135,50){\makebox(0,0){$\langle\id,c_1\rangle$}}
\put(15,10){\makebox(0,0){$\langle\id,c_0\rangle$}}
\put(135,10){\makebox(0,0){$\langle f,c_1\rangle$}}
\put(82,30){\makebox(0,0){$F$}}

\put(10,35){\vector(2,1){40}} \put(10,25){\vector(2,-1){40}}
\put(140,35){\vector(-2,1){40}} \put(140,25){\vector(-2,-1){40}}
\put(75,45){\vector(0,-1){30}}
\end{picture}
\end{center}
This is equivalent to the fact that $f$ is pseudo-isotopic to the
identity on $S^{d-1}$. Since it is not the case that for every $d$
every orientation preserving diffeomorphism of $S^{d-1}$ is
pseudo-isotopic to the identity (see \cite{KM63}, \cite{B67} and
\cite{C70}), we have that there is not always a unique
$d$-cobordism corresponding to $S^{d-1}\times I$, with chosen
ingoing and outgoing boundaries.

However, for $d\leq 6$ (and not only for these dimensions), every
orientation preserving diffeomorphism of $S^{d-1}$ is
pseudo-isotopic to the identity. This fact, for $d=2$, is
implicitly used by Kock, \cite{K03}, in order to pass from smooth
2-cobordisms to the pictures representing the underlying
manifolds. A result analogous to our Corollary \ref{cors1} holds
for 2-cobordisms of \cite{K03}.

\section{A Frobenius structure of spheres}\label{Frobenius}

In this section we follow the pattern given for $S^1$ in
\cite{A96} and \cite{K03} in order to define a Frobenius structure
for a sphere of any finite dimension.

For an oriented $d$-disc $D$ and its boundary $\partial D$, let
$\underline{\ed}$ be the $d$-cobordism $(D,\emptyset,(\partial
D))$ and let $\underline{\es}$ be the $d$-cobordism $(D,(\partial
D),\emptyset)$.

\begin{figure}[h!h!h!]
\centerline{\includegraphics[width=0.8 \textwidth]{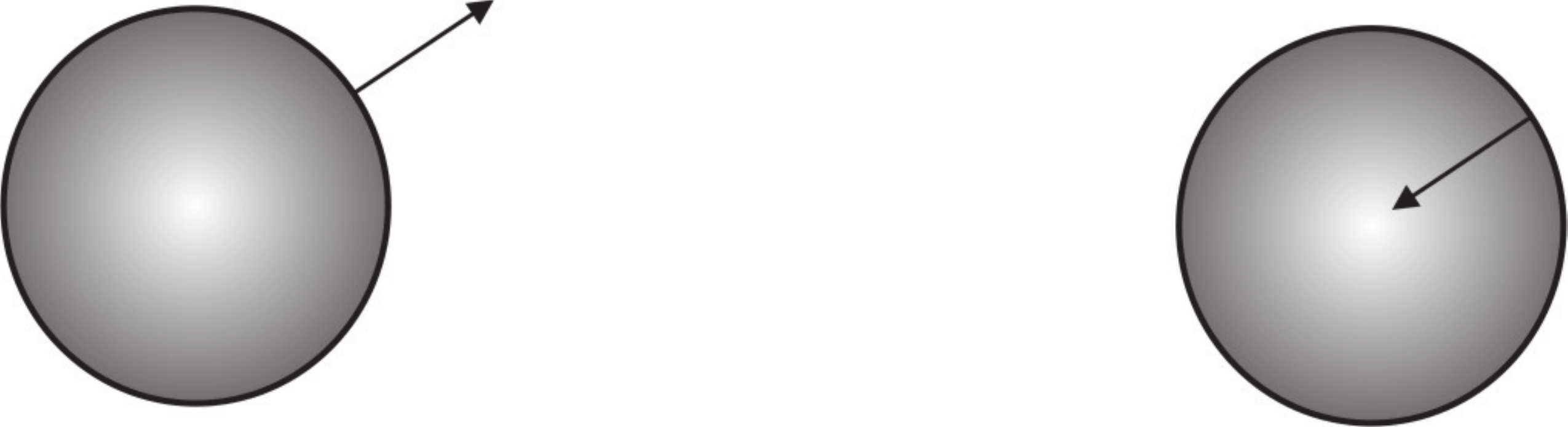}}
\caption{unit, counit} \label{slika1a}
\end{figure}

On the other hand, for $D_1$ and $D_2$ being two nonintersecting
$d$-discs in the interior of $D$, let $M$ be a $d$-manifold
obtained from $D$ by removing the interiors of $D_1$ and $D_2$. We
define $\underline{\md}$ to be the $d$-cobordism $(M,(\partial
D_1,\partial D_2),(\partial D))$ and $\underline{\ms}$ to be the
$d$-cobordism $(M,(\partial D),(\partial D_1,\partial D_2))$.

\begin{figure}[h!h!h!]
\centerline{\includegraphics[width=0.8 \textwidth]{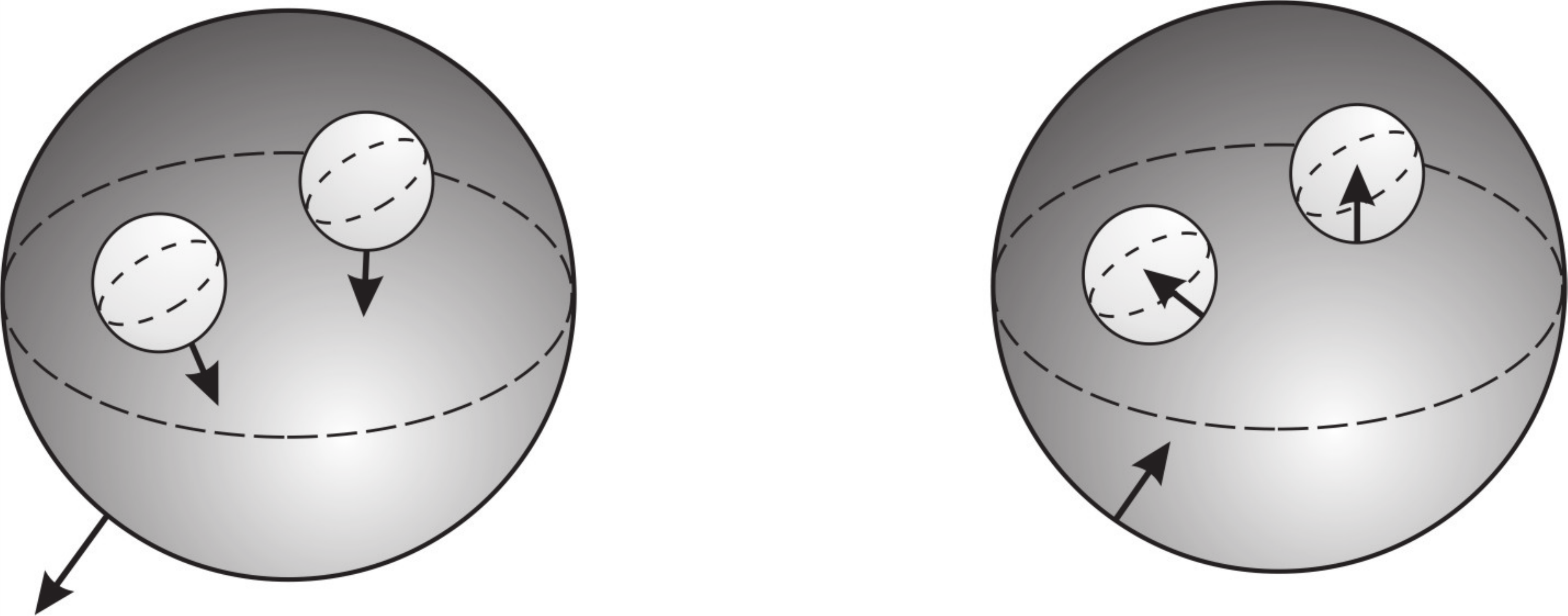}}
\caption{multiplication, comultiplication} \label{slika1b}
\end{figure}

It is not difficult to see that the above cobordisms, together
with the symmetric monoidal structure of $dCobS$, satisfy the
conditions necessary for $S^0$ to be a symmetric Frobenius object
of $1CobS$, and $S^{d-1}$, for $d\geq 2$ to be a commutative
Frobenius object of $dCobS$. For example, the equation (\ref{as}),
for $d=3$, is illustrated by the following picture.

\begin{figure}[h!h!h!]
\centerline{\includegraphics[width=0.8 \textwidth]{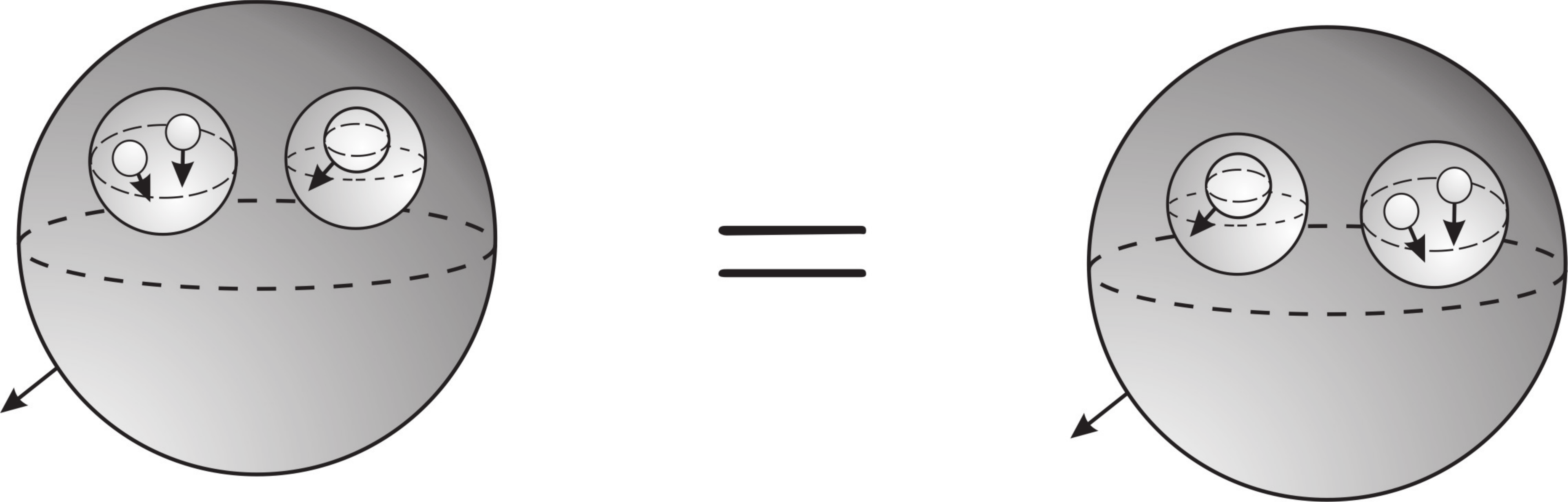}}
\caption{associativity} \label{slika1c}
\end{figure}

The defined Frobenius structure of $S^{d-1}$ guarantees that every
$d$TQFT maps this sphere to a Frobenius algebra. The image of
$S^{d-1}$ by a $d$TQFT is a commutative Frobenius algebra when
$d\geq 2$. This is a part of \cite[Proposition~1]{S95}, which is
essentially due to Dijkgraaf, \cite{D89}.

\section{Brauerian representation as a $1$TQFT}\label{Brauerian}

In this section, we pay attention to $1CobS$ in particular. We
show that Brauer, \cite{B37}, anticipated $1$TQFT by his matrix
representation of a class of diagrammatic algebras. When
restricted to $1CobS$, such a representation determines a matrix
Frobenius algebra as the image of the Frobenius object $S^{0}$.

Following the definition given in Section~\ref{dCobS}, the
category $1CobS$ has the objects
$\underline{0},\underline{1},\underline{2}, \ldots$, where
$\underline{0}$ is the empty set and $\underline{n}$ is the
$0$-dimensional manifold $\{-1,1,\ldots,3n-4,3n-2\}$ for which we
fix the orientation
\[
\varepsilon(x)=\left\{ \begin{array}{rl} 1, & x=3i-1, \\
-1, & x=3i+1.\end{array}\right.
\]
Hence, we may envisage an object of $1CobS$ as a finite sequence
built out of the pair $+-$. The arrows of $1CobS$ are the
equivalence classes of $1$-cobordisms. For example, the cobordism
$(M,f_0,f_1):\underline{2}\str \underline{1}$

\begin{center}
\begin{picture}(120,60)

\put(10,50){\vector(0,-1){35}}
\put(10,8){\makebox(0,0){$f_{0}(4)$}}
\put(10,55){\makebox(0,0){$f_{0}(-1)$}}
\put(40,50){\vector(0,-1){35}}
\put(40,8){\makebox(0,0){$f_{0}(1)$}}
\put(40,55){\makebox(0,0){$f_{1}(1)$}}
\put(70,50){\vector(0,-1){35}}
\put(70,8){\makebox(0,0){$f_{1}(-1)$}}
\put(70,55){\makebox(0,0){$f_{0}(2)$}}

\put(100,33){\circle{30}} \put(116,35){\vector(0,0){0}}
\put(0,33){\makebox(0,0){$M$}}

\end{picture}
\end{center}
is illustrated by the following picture

\begin{center}
\begin{picture}(120,60)

\qbezier(5,46)(35,10)(75,46) \put(53,17){\vector(-1,1){30}}
\put(27,17){\line(1,1){30}} \put(29,19){\vector(-1,-1){3}}
\put(72,43){\vector(1,1){3}}

\put(100,33){\circle{30}} \put(116,35){\vector(0,0){0}}

\put(1,50){$+$} \put(20,50){$-$} \put(54,50){$+$} \put(72,50){$-$}
\put(22,8){$+$} \put(50,8){$-$}

\end{picture}
\end{center}

The category $1CobS$ is skeletal. If there is an isomorphism $K$
between $\underline{n}$ and $\underline{m}$, then it is easy to
see that there are no \emph{cup} components in $K$, i.e.
components presented by
\begin{center}
\begin{picture}(140,30)(0,10)

\qbezier(10,30)(30,5)(50,30) \put(30,18){\vector(1,0){3}}
\put(6,35){$+$} \put(47,35){$-$}

\put(70,15){or}

\qbezier(90,30)(110,5)(130,30) \put(110,18){\vector(-1,0){3}}
\put(86,35){$-$} \put(127,35){$+$}

\end{picture}
\end{center}
Otherwise, there would be such components in $K^{-1} \circ K:
\underline{n}\str \underline{n}$, which is impossible.
Analogously, there are no \emph{cap} components in $K$, hence
$n=m$, which implies $\underline{n}=\underline{m}$ (cf.\
Proposition~\ref{isomorphism}).

\vspace{2ex}

For the infinite sequence $-1,1,2,4,\ldots,3i-1, 3i+1,\ldots$ let
$\underline{\underline{n}}$ denote the set of its first $n$
members. In order to obtain a symmetric strict monoidal category
containing $1CobS$ as a full subcategory, let $1Cob$ be the
category whose set of objects is
$$\{(\underline{\underline{n}},\varepsilon)\mid n \in N, \varepsilon:
\underline{\underline{n}}\str \{-1,1\}\}$$ and whose arrows are
the equivalence classes of cobordisms of the form
$$(M,f_0: (\underline{\underline{n}},\varepsilon_{0})\str M,f_1:
(\underline{\underline{m}},\varepsilon_{1})\str M),$$ where $M$ is
a $1$-manifold such that its boundary $\partial M$ is a disjoint
union of $\Sigma_0$ and $\Sigma_1$, and $f_0$ is an orientation
preserving embedding whose image is $\Sigma_0$, while $f_1$ is an
orientation reversing embedding whose image is $\Sigma_1$. The
symmetric monoidal structure of $1Cob$ is defined as for $1CobS$
by ``putting side by side'' and by using the symmetry defined in
an analogous way as $\underline{\tau_{n,m}}$ defined in
Section~\ref{dCobS}. A connected component of $M$ homeomorphic to
$S^{1}$ is called \emph{circular} component of the cobordism.
Again, as in Corollary \ref{cors1}, every arrow of $1Cob$ is
completely determined by a $1$-manifold $M$ and two sequences
$\Sigma_0=(\Sigma_0^0,\ldots,\Sigma_0^{n-1})$ and
$\Sigma_1=(\Sigma_1^0,\ldots,\Sigma_1^{m-1})$ of points, one of
the ingoing boundary and the other of the outgoing boundary. The
category $1Cob$ is not skeletal since we have two different
objects $(\underline{\underline{2}},\varepsilon_{0})$ and
$(\underline{\underline{2}},\varepsilon_{1})$ with
$\varepsilon_{0}(-1)=1,\varepsilon_{0}(1)=-1,\varepsilon_{1}(-1)=-1,
\varepsilon_{1}(1)=1$, which are isomorphic via symmetry.

Brauer, \cite{B37}, introduced a class of diagrammatic algebras
and found their matrix representation. In \cite[Section~6]{DP003}
a generalization of this representation to a category of diagrams
is given (see also \cite{DP03} and \cite[Section~14]{DP12}). This
generalization leads to the following assignment of matrices to
the arrows of $1Cob$.

Let $\mathcal{F}$ be a field of characteristic $0$ and let $p$ be
a natural number greater than or equal to $2$. For an arrow
$K=(M,\Sigma_0,\Sigma_1):(\underline{\underline{n}},
\varepsilon_{0}) \str (\underline{\underline{m}},\varepsilon_{1})$
of $1Cob$, let $\rho_K$ be the following equivalence relation on
the disjoint union $(n\times \{0\})\cup (m\times\{1\})$ of finite
ordinals $n=\{0,\ldots,n-1\}$ and $m=\{0,\ldots,m-1\}$. For
$(i,k)$ and $(j,l)$ elements of $(n\times \{0\})\cup
(m\times\{1\})$, we have that $(i,k)\rho_K (j,l)$
\[
\mbox{\rm when the points } \Sigma_k^i\; \mbox{\rm and }
\Sigma_l^j \; \mbox{\rm belong to the same connected component of
} M.
\]

For every $K:(\underline{\underline{n}}, \varepsilon_{0}) \str
(\underline{\underline{m}},\varepsilon_{1})$ we define a matrix
$A(K) \in \mathcal{M}_{p^{m}\times p^{n}}$ in the following way.
For $a_0$ such that $0\leq a_{0} < p^{n}$, which denotes a column
of $A(K)$, and $a_1$ such that $0\leq a_{1} < p^{m}$, which
denotes a row of $A(K)$, write $a_{0}$ in the base $p$ system with
$n$ digits $a_{0}^{0}\ldots a_{0}^{n-1}$, and $a_{1}$ in the base
$p$ system with $m$ digits $a_{1}^{0}\ldots a_{1}^{m-1}$. For
example, if $p=2$, $n=5$, $m=3$, $a_{0}=10$, $a_{1}=5$, we have
$a_{0}=01010$ and $a_{1}=101$.

We define the $(a_{1},a_{0})$ element of $A(K)$ to be $1$ when for
every $(i,k)$ and $(j,l)$ from $(n\times \{0\})\cup
(m\times\{1\})$ we have that
\[
(i,k)\rho_K (j,l) \; \Rightarrow a_{k}^{i}=a_{l}^{j};
\]
otherwise it is $0$.

If we take $K$ to be given by the following picture,
\begin{center}
\begin{picture}(100,40)

\qbezier(50,32)(70,10)(95,32) \put(32,7){\vector(-1,1){25}}
\put(8,7){\vector(1,1){25}} \put(50,7){\line(1,1){25}}
\put(2,35){$-$} \put(30,35){$-$} \put(46,35){$+$} \put(72,35){$+$}
\put(92,35){$-$}

\put(2,-2){$-$} \put(30,-2){$-$} \put(46,-2){$+$}

\put(71,21){\vector(1,0){3}} \put(52,9){\vector(-1,-1){3}}
\end{picture}
\end{center}
and we take $p=2$ as above, then the $(5,10)$ element of $A(K)$ is
$1$ since the sequences $01010$ and $101$ ``match'' into the
picture of $\rho_K$.

\begin{center}
\begin{picture}(100,40)

\qbezier(50,32)(70,10)(95,32) \put(32,7){\line(-1,1){25}}
\put(8,7){\line(1,1){25}} \put(50,7){\line(1,1){25}}
\put(2,35){$0$} \put(32,35){$1$} \put(46,35){$0$} \put(72,35){$1$}
\put(92,35){$0$}

\put(2,-3){$1$} \put(32,-3){$0$} \put(46,-3){$1$}

\end{picture}
\end{center}

Let $\mathbf{Mat}_{\mathcal{F}}$ be the category whose objects are
vector spaces $\mathcal{F}^{n}$, $n\geq 1$, and whose arrows from
$\mathcal{F}^{n}$ to $\mathcal{F}^{m}$ are $m \times n$ matrices
over the field $\mathcal{F}$. The identity matrix of order $n$ is
the identity arrow on $\mathcal{F}^{n}$ and matrix multiplication
is the composition of arrows. One can identify the objects of
$\mathbf{Mat}_{\mathcal{F}}$ with natural numbers (the dimensions
of vector spaces) as it was done in \cite{DP12}. The category
$\mathbf{Mat}_{\mathcal{F}}$ may be considered as a skeleton of
the category $\mathbf{Vect}_{\mathcal{F}}$ of finite-dimensional
vector spaces over $\mathcal{F}$.  Hence,
$\mathbf{Mat}_{\mathcal{F}}$ and $\mathbf{Vect}_{\mathcal{F}}$ are
equivalent.

The category $\mathbf{Mat} _{\mathcal{F}}$ is symmetric strict
monoidal with respect to the multiplication on objects considered
as natural numbers, and the Kronecker product on arrows
(matrices). The symmetry is brought by the family of $nm \times
mn$ permutation matrices $S_{n,m}$. The matrix $S_{n,m}$ is the
matrix representation of the linear map $\sigma:
\mathcal{F}^{n}\otimes \mathcal{F}^{m} \rightarrow
\mathcal{F}^{m}\otimes \mathcal{F}^{n}$ with
 respect to the standard ordered bases, defined on the basis vectors
by $\sigma(e_{i}\otimes f_{j})=f_{j}\otimes e_{i}$. For example,
$S_{3,2}$ is the matrix
$$
    {\begin{bmatrix}
        1 & 0 & 0 & 0 & 0 & 0\\
        0 & 0 & 1 & 0 & 0 & 0\\
        0 & 0 & 0 & 0 & 1 & 0\\
        0 & 1 & 0 & 0 & 0 & 0\\
        0 & 0 & 0 & 1 & 0 & 0\\
        0 & 0 & 0 & 0 & 0 & 1\\
    \end{bmatrix}}
$$

Consider the following functor $B$ from $1Cob$ to $\mathbf{Mat}_
{\mathcal{F}}$. On objects it is defined by
$B(\underline{\underline{n}} ,\varepsilon)=p^{n}$ and on arrows we
define it as
\[
B(K)=p^{c}\cdot A(K),
\]
where $c$ is the number of circular components of $K$, and $A(K)$
is the ${0-1}$ matrix defined above. That this is indeed a functor
stems from \cite[Section~5, Proposition~4]{DP003} and that it is
faithful stems from \cite[Section~14]{DP12}. We shall not go here
into any more detail about this matter.

In order to conclude that this functor is monoidal, note that for
matrices $X \in \mathcal{M}_{m \times n}$ and $Y \in
\mathcal{M}_{k \times l}$ we have $Z=X\otimes Y \in
\mathcal{M}_{(m\cdot k) \times (n \cdot l)}$ and $$x_{i,j}\cdot
y_{q,r}=z_{i\cdot k + q, j \cdot l + r}.$$ If $K$ is obtained from
$K_{1}$ and $K_{2}$ by ``putting side by side'' and $Z$ is the
matrix $A(K)$, while $X$ and $Y$ are $A(K_{1})$ and $A(K_{2})$
respectively, then
\[
z_{i\cdot k + q, j \cdot l + r}=1\;\; \mbox{\rm iff }\;\;
x_{i,j}=y_{q,r}=1.
\]
In our example for $K_{1}$ and $K_{2}$, respectively being
\begin{center}
\begin{picture}(130,40)

\qbezier(90,32)(110,10)(135,32) \put(32,7){\vector(-1,1){25}}
\put(8,7){\vector(1,1){25}} \put(90,7){\line(1,1){25}}
\put(2,35){$-$} \put(30,35){$-$} \put(86,35){$+$}
\put(112,35){$+$} \put(132,35){$-$}

\put(2,-2){$-$} \put(30,-2){$-$} \put(86,-2){$+$}

\put(111,21){\vector(1,0){3}} \put(92,9){\vector(-1,-1){3}}
\put(60,20){\makebox(0,0){and}}
\end{picture}
\end{center}
we have $z_{5,10}=x_{2,1} \cdot y_{1,2}$.

It is easy to check that $B$ maps symmetry to symmetry.
Consequently, the functor $B$ may be said to be a $1$TQFT.

Let us now restrict the functor $B$ to the category $1CobS$. Since
$S^0$, i.e.\ the object $\underline{1}$ is equipped with a
Frobenius structure in $1CobS$, consequently in $1Cob$, the image
of $\underline{1}$ by the monoidal functor $B$ is a Frobenius
algebra. It is interesting that $B$ brings to $B(\underline{1})$
the structure of a matrix Frobenius algebra (for the notion of
matrix Frobenius algebra see \cite[2.2.16]{K03}).

Note that $B(\underline{1})$ is $p^2$, i.e.\ the vector space
$\mathcal{F}^{p^2}$. Every vector
\[
\vec{v}=\left[\!\!
\begin{array}{c} v_0 \\ \vdots \\ v_{p^2-1} \end{array}
\!\!\right]\in \mathcal{F}^{p^2}
\]
corresponds to the matrix $H(\vec{v})\in \mathcal{M}_{p\times p}$
whose $(i,j)$ member is $v_{i\cdot p+j}$. This is the standard
isomorphism $H:\mathcal{F}^{p^2}\str\mathcal{M}_{p\times p}$. In
order to show that $B$ brings the structure of a matrix Frobenius
algebra to $\mathcal{M}_{p\times p}=B(\underline{1})$, it suffices
to show that $B(\underline{\md})$ represents the multiplication of
matrices and that $B(\underline{\es})$ represents the trace form.

The arrow $\underline{\md}:\underline{2}\str \underline{1}$ of
$1Cob$ is presented by the following picture
\begin{center}
\begin{picture}(80,40)

\qbezier(25,32)(40,10)(55,32) \put(30,7){\line(-1,1){25}}
\put(50,7){\vector(1,1){25}} \put(1,35){$+$} \put(21,35){$-$}
\put(51,35){$+$} \put(72,35){$-$} \put(26,0){$+$} \put(46,0){$-$}

\put(41,21){\vector(-1,0){3}} \put(27,10){\vector(1,-1){3}}

\end{picture}
\end{center}
and the corresponding matrix $B(\underline{\md})$ is in
$\mathcal{M}_{p^2\times p^4}$. Our goal is to show that for the
standard isomorphism
\[
H_2:\mathcal{F}^{p^4}\str\mathcal{M}_{p^2\times p^2}
\]
defined as $H$ above (i.e.\ $(i,j)$ member of $H_2(\vec{v})$ is
$v_{i\cdot p^{2}+j}$) and arbitrary matrices
$X,Y\in\mathcal{M}_{p\times p}$ we have that
\[
H(B(\underline{\md})\,H_2^{-1}\!(X\otimes Y))=XY.
\]

When $p=2$, the matrix $B(\underline{\md})$ is
\[
\left[\!\!
\begin{array}{cccccccccccccccc} 1 & 0 & 0 & 0 & 0 & 0 & 1 & 0 & 0 & 0 & 0 & 0 & 0 & 0 & 0 & 0 \\
0 & 1 & 0 & 0 & 0 & 0 & 0 & 1 & 0 & 0 & 0 & 0 & 0 & 0 & 0 & 0 \\
0 & 0 & 0 & 0 & 0 & 0 & 0 & 0 & 1 & 0 & 0 & 0 & 0 & 0 & 1 & 0 \\
0 & 0 & 0 & 0 & 0 & 0 & 0 & 0 & 0 & 1 & 0 & 0 & 0 & 0 & 0 & 1
\end{array} \!\!\right]
\]
and $H_2^{-1}(X\otimes Y)$ is the vector
\[
\vec{v}=\left[\!\!
\begin{array}{c} v_0 \\ \vdots \\ v_{15} \end{array}
\!\!\right]\in \mathcal{F}^{16}
\]
where $v_0=x_{00}\cdot y_{00}$, $v_1=x_{00}\cdot y_{01}$,
$v_6=x_{01}\cdot y_{10}$, $v_7=x_{01}\cdot y_{11}$,
$v_8=x_{10}\cdot y_{00}$, $v_{9}=x_{10}\cdot y_{01}$,
$v_{14}=x_{11}\cdot y_{10}$ and  $v_{15}=x_{11}\cdot y_{11}$.
Hence, $B(\underline{\md})\,H_2^{-1}\!(X\otimes Y)$ is
\[
\left[\!\!
\begin{array}{c} x_{00}\cdot y_{00}+x_{01}\cdot y_{10} \\ x_{00}\cdot y_{01}+x_{01}\cdot y_{11}
\\ x_{10}\cdot y_{00}+x_{11}\cdot y_{10} \\ x_{10}\cdot y_{01}+x_{11}\cdot y_{11} \end{array}
\!\!\right]
\]
which is mapped to $XY$ by $H$.

For the general case, let
$\vec{u}=B(\underline{\md})\,H_2^{-1}\!(X\otimes Y)$ and
$A=H(\vec{u})$. We want to show that for $0\leq i,j\leq p-1$,
\[
a_{i,j}=\sum_{k=0}^{p-1} x_{i,k}\cdot y_{k,j}.
\]
Since the element $a_{i,j}$ is equal to $u_{i\cdot p+j}$, we are
interested in the ${(i\cdot p+j)}$-th row of the matrix
$B(\underline{\md})$. In this row, which in the base $p$ system is
presented by the sequence $ij$, the entry $1$ occurs $p$ times in
the columns presented in the base $p$ system by the sequences
\[
i 0 0 j,\quad i11j,\quad \ldots \quad ikkj, \quad \ldots \quad i
(p-1) (p-1) j,
\]
and all the other elements are $0$. The column presented by $ikkj$
is actually the $(i\cdot p^{3}+k\cdot p^{2}+k \cdot p+j)$-th
column of the matrix $B(\underline{\md})$. Since the corresponding
row of $H_2^{-1}\!(X\otimes Y)$ is equal to $x_{i,k} \cdot
y_{k,j}$, we have that
\[
a_{i,j}=u_{i\cdot p+j}=\sum_{k=0}^{p-1} x_{i,k}\cdot y_{k,j}.
\]

The arrow $\underline{\es}:\underline{1}\str \underline{0}$ of
$1Cob$ is presented by the following picture
\begin{center}
\begin{picture}(60,30)(0,10)

\qbezier(10,30)(30,5)(50,30) \put(30,18){\vector(1,0){3}}
\put(6,35){$+$} \put(47,35){$-$}

\end{picture}
\end{center}
and the corresponding matrix $B(\underline{\es})$ is in
$\mathcal{M}_{1\times p^2}$. Our goal is to show that for an
arbitrary matrix $X\in\mathcal{M}_{p\times p}$ we have that
\[
B(\underline{\es})\,H^{-1}\!(X)= \mbox{\rm tr} (X).
\]

When $p=2$, this equality reads
\[
\left[\!\!
\begin{array}{cccc} 1 & 0 & 0 & 1
\end{array} \!\!\right] \left[\!\!
\begin{array}{c} x_{00} \\ x_{01} \\ x_{10} \\ x_{11}
\end{array} \!\!\right]= x_{00}+x_{11}.
\]
For the general case, in the row of the matrix
$B(\underline{\es})$ the entry $1$ occurs $p$ times in the columns
presented in the base $p$ system by the sequences
\[
00,\quad 11,\quad \ldots \quad kk, \quad \ldots \quad
 (p-1)(p-1),
\]
and all the other elements are $0$. The column presented by $kk$
is actually the $(k \cdot p+k)$-th column of the matrix
$B(\underline{\es})$. Since the corresponding row of $H^{-1}\!(X)$
is equal to $x_{k,k}$, we have that
\[
B(\underline{\es})\,H^{-1}\!(X)=\sum_{k=0}^{p-1} x_{k,k}.
\]

\section{The category $\mathbf{K}$}\label{categoryK}

Our intention is to define the category $\mathbf{K}$ as a PROP, in
the sense of \cite[Chapter~V]{ML65}, having 1 as the universal
commutative Frobenius object, in the same sense as 1, as an object
of the simplicial category $\Delta$ is the universal monoid. The
category $\Delta$ is introduced in \cite[Section~VII.5]{ML71} as
the concrete category of monotone functions between finite
ordinals. Alternatively, this category could be introduced in a
pure syntactical manner by generators and relations, via
\cite[Proposition~2, Section~VII.5]{ML71}.

We choose this alternative approach and present the category
$\mathbf{K}$ by generators and relations. In this way we stipulate
the intended universal property in its definition.

More formally, consider the category $\mathcal{F}$ whose objects
are symmetric strict monoidal categories with one distinguished
commutative Frobenius object and whose arrows are symmetric
monoidal functors preserving distinguished objects and their
Frobenius structures. Since the notions of symmetric strict
monoidal category and commutative Frobenius object are purely
equational, the forgetful functor $G$ from $\mathcal{F}$ to the
category $\set$ of sets and functions, which maps an object of
$\mathcal{F}$, i.e.\ a symmetric monoidal category, to the set of
its objects, has a left adjoint $F$. As in universal algebra,
$FX$, for a set $X$ is built out of a term model. Our category
$\mathbf{K}$ is $F\emptyset$. What follows is a brief description
of our construction of $\mathbf{K}$ and we refer to
Section~\ref{appendixK} for details.

The category $\mathbf{K}$ has the set of finite ordinals $\omega$
as the set of objects. The ordinal $n$ is interpreted as the
$n$-th tensor power of the distinguished commutative Frobenius
object. Hence, the monoidal structure on objects is given by
addition. In order to define the arrows of this category, an
equational system is introduced in Section~\ref{appendixK}.

Briefly, as in every syntactical construction of a free object,
words built out of $\mj$, $\circ$, $\otimes$, $\tau$, $\md$,
$\ed$, $\ms$ and $\es$ denoting the arrows of $\mathbf{K}$ are
introduced. We call these words \emph{terms}. Every such a term
has its source and target. The terms are quotient by the smallest
equivalence relation guaranteeing that 1 is a commutative
Frobenius object in $\mathbf{K}$. (See Section~\ref{appendixK} for
details.) The equivalence class of a term $f$ is denoted by $[f]$
and $\{[f]\mid f\;\mbox{\rm is a term}\}$ is the set of arrows of
$\mathbf{K}$. The source of $[f]$ is the source of $f$ and the
same holds for targets. The identity arrow on $n$ is $[\mj_n]$ and
$[g]\circ[f]$ is $[g\circ f]$.

The category $\mathbf{K}$ is strict monoidal with respect to the
monoidal structure given by $\otimes$ and 0. Its symmetry is given
by the family of $\tau$ arrows. It is \emph{skeletal} by
Corollary~\ref{skeletal}.

The category $\mathbf{K}$, since it is the image of the initial
object in $\set$ under the functor $F$, has the following
universal property: for every commutative Frobenius object $M$ in
a symmetric strict monoidal category $\mathcal{M}$, there is a
unique symmetric monoidal functor $U:\mathbf{K}\str \mathcal{M}$
such that $U(1)=M$, and $U$ preserves the Frobenius structure.
Hence, for $d\geq 2$, there is a unique symmetric monoidal functor
from $\mathbf{K}$ to $dCobS$ that maps 1 to $\underline{1}$. We
call this functor the \emph{interpretation} of $\mathbf{K}$ in
$dCobS$. That the interpretation is faithful is shown in
Section~\ref{faithfulness}.

The equations (\ref{ct}) (see Section~\ref{appendixK}) are usually
not mentioned in the calculations that follow. Hence, we omit
parenthesis tied to nested compositions, and erase or add
compositions with identities, when necessary.

\begin{rem}
We could start with the category $2CobS$ instead of $\mathbf{K}$
(cf.\ Corollary~\ref{universal}), which would be more in the style
of the definition of the simplicial category given in
\cite[Section~VII.5]{ML71}. However, for the proof of our main
result, if we relied on 2CobS instead on K, then we would miss the
syntax necessary for our approach. This would lead to a certain
amount of imprecision.
\end{rem}

\section{Normal form for arrows of $\mathbf{K}$}\label{normal
form}

In this section, we define a normal form for terms and show that
every arrow of $\mathbf{K}$ is representable by a term in normal
form. This normal form is essentially the same as the one given in
\cite[1.4.16]{K03}. The normal form is then used in
Section~\ref{faithfulness} for the proof of faithfulness of the
interpretation. Some proofs are illustrated by pictures
corresponding to the interpretation of $\mathbf{K}$ in $2CobS$.

We start with some auxiliary notions. Let $V_{-1}=\ed$,
$\Lambda_{-1}=\es$, $V_0=H_0=\Lambda_0=\mj_1$, and for $n\geq 1$,
let
\[
V_n=\md\circ(\md\otimes\mj_1)\circ\ldots\circ(\md\otimes\mj_{n-1}):n+1\str
1,
\]
\[
H_n=\underbrace{(\md\circ\ms)\circ\ldots\circ(\md\circ\ms)}_{n}:1\str
1,
\]
\[
\Lambda_n=(\ms\otimes\mj_{n-1})\circ\ldots\circ(\ms\otimes\mj_1)\circ\ms:1\str
n+1.
\]
With the help of these terms, for $n,m,p\geq 0$, we define
$E_{p,m,n}$ as
\[
\Lambda_{p-1}\circ H_m\circ V_{n-1}:n\str p
\]

A term is a $\tau$-\emph{term} when $\md$, $\ed$, $\ms$ and $\es$
do not occur in it. For every $\tau$-term $f:n\str n$ there exists
a unique permutation on $n$ that corresponds to~$f$.

A term is \emph{special} when it is a $\tau$-term, or for $k\geq
1$, it is of the form
\[
\pi\circ \bigotimes_{i=1}^k E_{p_i,m_i,n_i} \circ \chi,
\]
where $\pi$ and $\chi$ are $\tau$-terms. We call $\chi$, the
\emph{head}, $\bigotimes_{i=1}^k E_{p_i,m_i,n_i}$, the
\emph{center}, and $\pi$, the \emph{tail} of this term.

\begin{prop}\label{tv1}
Every term is equal to a special term.
\end{prop}

We use the following lemmata in the proof of
Proposition~\ref{tv1}.

\begin{lem}\label{lem1}
Every term is equal to a term of the form $f_n\circ\ldots\circ
f_0$, $n\geq 0$, where every $f_i$ is of the form
$\mj_l\otimes\beta\otimes\mj_r$, for $l,r\geq 0$ and
$\beta\in\{\tau,\ms,\es,\ed,\md\}$.
\end{lem}

\begin{proof}
By relying on the equations
\[
f_1\otm f_2=(f_1\otm\mj_{m_2})\circ (\mj_{n_1}\otm
f_2)\quad\mbox{\rm and}\quad (g\circ f)\otimes
\mj_m=(g\otimes\mj_m)\circ(f\otimes \mj_m),
\]
derived from (\ref{ct}) and (\ref{fn}).
\end{proof}

\begin{lem}\label{lem2}
For every permutation on $n$, there is a $\tau$-term $\pi:n\str n$
such that this permutation corresponds to $\pi$. If the
permutations corresponding to two $\tau$-terms are equal, then
these terms are equal in $\mathcal{K}$.
\end{lem}
\begin{proof}
By symmetric monoidal coherence (see \cite{ML63}).
\end{proof}
From now on, we identify a $\tau$-term with the corresponding
permutation.

\begin{lem}\label{lem3}
For every $\tau$-term $\pi:p\str p$ and every $l\in p$, there is a
$\tau$-term $\pi':p-1\str p-1$ such that for $j=\pi^{-1}(l)$,
$\pi$ is equal to
\[
(\tau_{1,l}\otm \mj_{p-l-1})\circ(\mj_1\otm \pi')\circ
(\tau_{j,1}\otm \mj_{p-j-1}).
\]
\end{lem}
\begin{proof}
The permutation corresponding to
\[
(\tau_{l,1}\otm \mj_{p-l-1})\circ\pi\circ (\tau_{1,j}\otm
\mj_{p-j-1})
\]
has 0 as a fix point. Hence, by Lemma~\ref{lem2}, there is a
$\tau$-term $\pi'$ such that this permutation corresponds to
$\mj_1\otm \pi'$. By Lemma~\ref{lem2} and (\ref{iv}) this
concludes the proof.
\end{proof}

By relying on (\ref{fn}), (\ref{ca}) and (\ref{cu}), we obtain the
following two lemmata.

\begin{figure}[h!h!h!]
\centerline{\includegraphics[width=0.6 \textwidth]{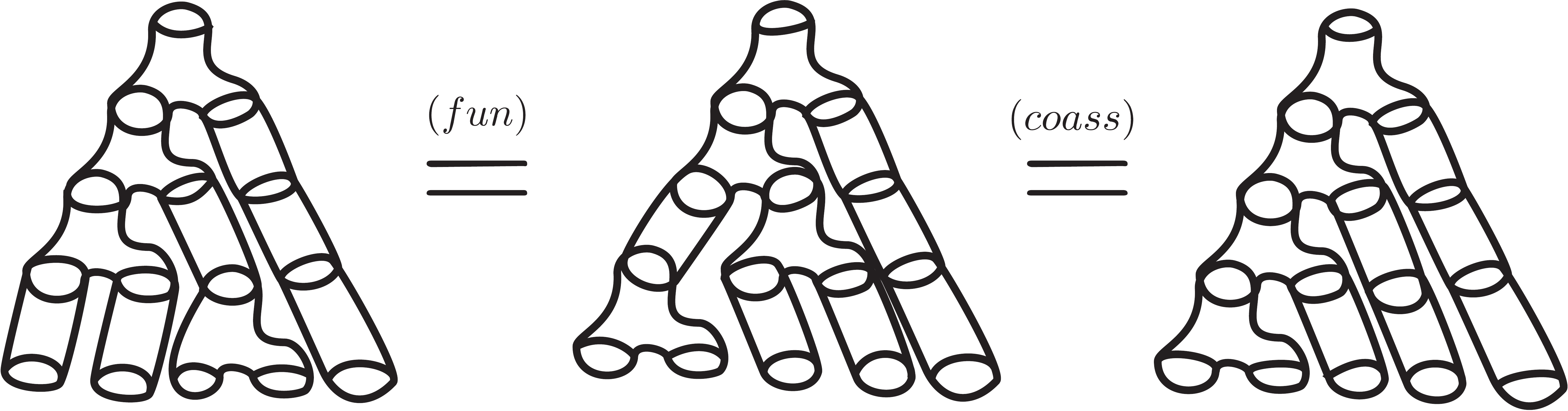}}
\caption{Lemma \ref{lem4}} \label{lemas5}
\end{figure}

\begin{lem}\label{lem4}
For $l+r=n\geq 0$, we have
$(\mj_l\otimes\ms\otimes\mj_r)\circ\Lambda_n=\Lambda_{n+1}$.

\end{lem}

\begin{lem}\label{lem5}
For $l+r=n\geq 0$, we have
$(\mj_l\!\otimes\es\!\otimes\!\mj_r)\circ\Lambda_n=\!\Lambda_{n-1}$.
\begin{figure}[h!h!h!]
\centerline{\includegraphics[width=0.8 \textwidth]{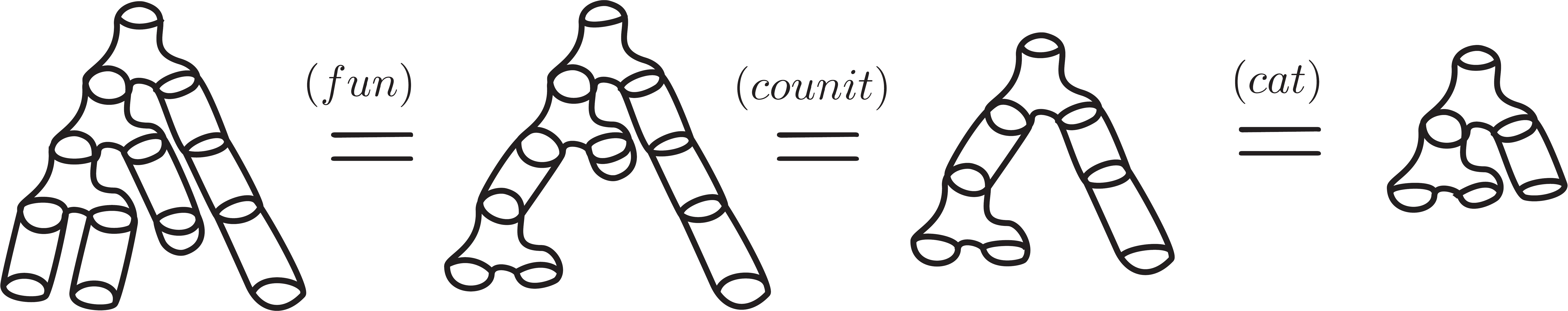}}
\caption{Lemma \ref{lem5}} \label{lemas6}
\end{figure}
\end{lem}

\begin{lem}\label{lem6}
For every $\tau$-term $\pi:l+r\str l+r$, we have
\[
(\mj_l\!\otimes\ed\otimes \mj_r)\circ \pi=
(\tau_{1,l}\otm\mj_r)\circ(\mj_1\otm\pi)\circ (\ed\otm\mj_{l+r}).
\]
\end{lem}
\begin{proof}
We prove this from right to left, by relying on (\ref{fn}),
(\ref{nt}) and the equation $\tau_{0,l}=\mj_l$, which is derivable
by (\ref{iv}), (\ref{hx}) and (\ref{st}) as follows
\begin{tabbing}
\hspace{1.5em}\= $\mj_l$ \= $=\tau_{l,0}\circ
\tau_{0,l}=\tau_{l,0}\circ\tau_{0+0,l}=\tau_{l,0}\circ(\tau_{0,l}\otm\mj_0)\circ
(\mj_0\otm\tau_{0,l})$
\\[1ex]
\> \> $=\tau_{l,0}\circ\tau_{0,l}\circ \tau_{0,l}= \tau_{0,l}$.
\end{tabbing}

\vspace{-3.5ex}
\end{proof}

\begin{figure}[h!h!h!]
\centerline{\includegraphics[width=0.8 \textwidth]{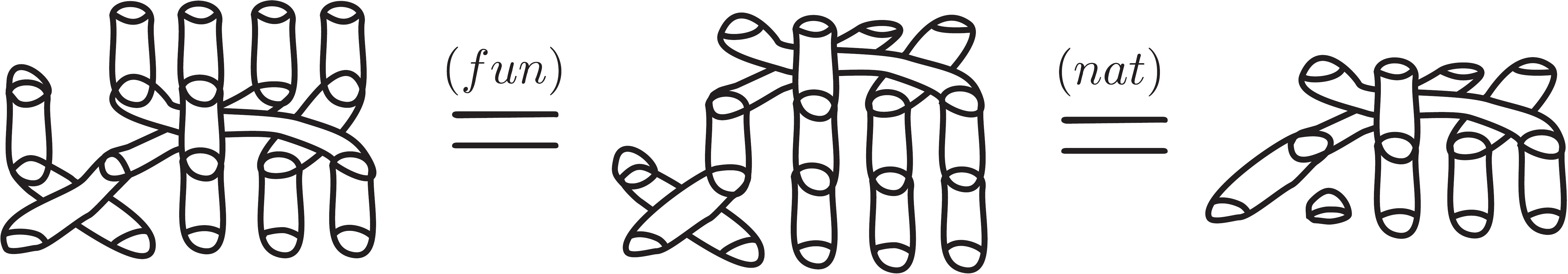}}
\caption{Lemma \ref{lem6}} \label{lemas7}
\end{figure}

By relying on (\ref{ca}), (\ref{as}), (\ref{cocm}), (\ref{cm}),
and the fact that every permutation is equal to a composition of
transpositions, we can prove the following.

\begin{lem}\label{lem7}
For every $\tau$-term $\pi:n+1\str n+1$, we have $\pi\circ
\Lambda_n= \Lambda_n$, and $V_n\circ \pi= V_n$.
\end{lem}

\begin{figure}[h!h!h!]
\centerline{\includegraphics[width=0.6 \textwidth]{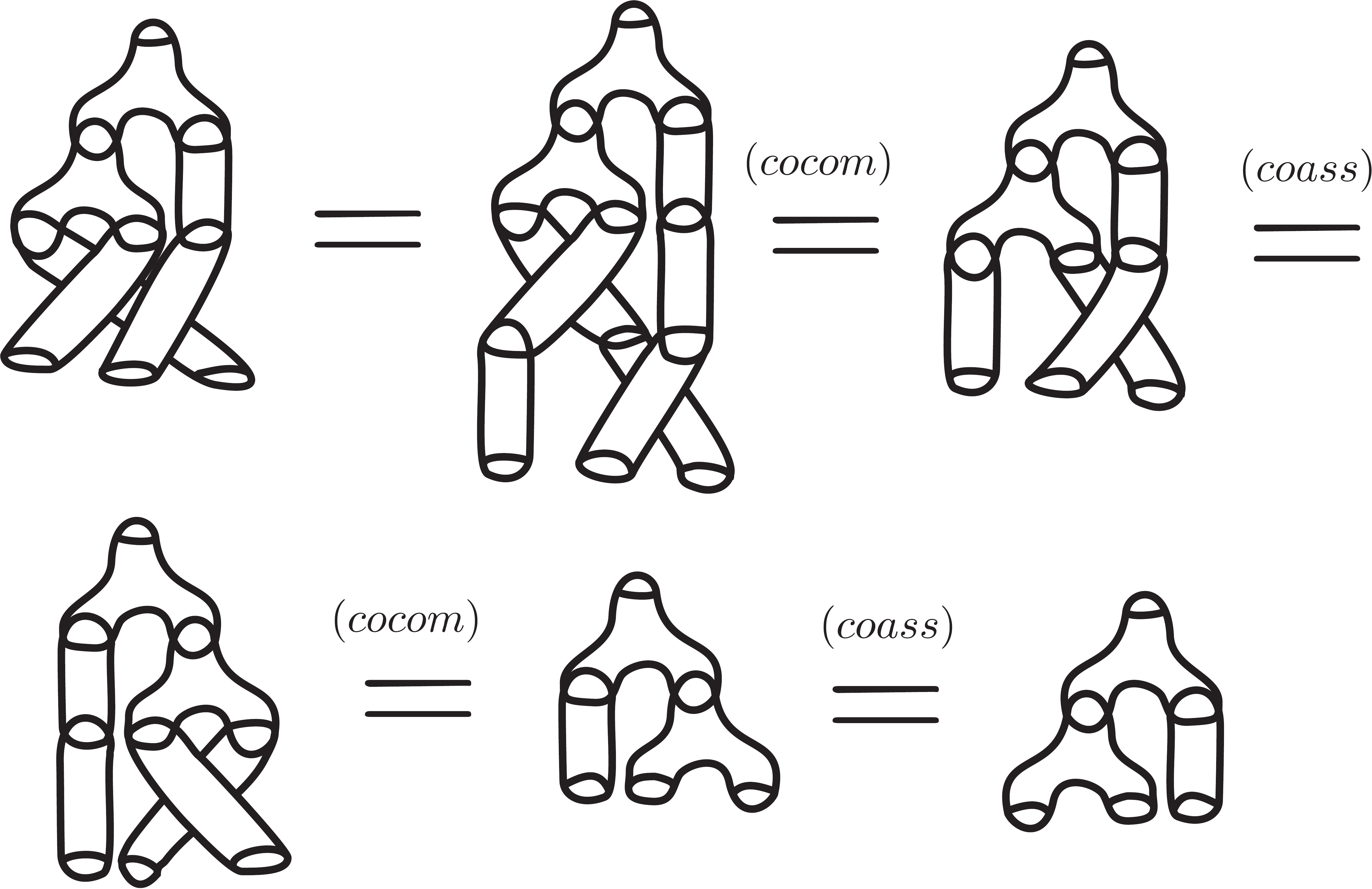}}
\caption{Lemma \ref{lem7}} \label{lemas8}
\end{figure}

For a $\tau$-term $\pi:p\str p$ with $p\geq2$, we say that
$l,l+1\in p$ are \emph{parallel} in $\pi$ when
$\pi^{-1}(l+1)=\pi^{-1}(l)+1$, i.e.\ for some $j\in p$, $\pi(j)=l$
and $\pi(j+1)=l+1$.

\begin{lem}\label{lem8}
For a special term $f$, which is not a $\tau$-term, with the
target $p\geq 2$, and every $l\in p-1$, there is a special term
equal to $f$, such that $l,l+1$ are parallel in its tail.
\end{lem}
\begin{proof}
Let $f$ be $\pi\circ \bigotimes_{i=1}^k E_{p_i,m_i,n_i} \circ
\chi$. If $l$ and $l+1$ are tied by $\pi$ to the target of one $E$
in the center of $f$, i.e. there is $j\in \{1,\ldots,k\}$ such
that
\[
\sum_{i=1}^{j-1} p_i\leq \pi^{-1}(l), \pi^{-1}(l+1)<
\sum_{i=1}^{j} p_i,
\]
then, if necessary, by Lemma~\ref{lem7} a $\tau$-term could be
added in between the tail and the center of $f$ in order to obtain
a new tail such that $l,l+1$ are parallel in it.

\begin{figure}[h!h!h!]
\centerline{\includegraphics[width=0.8
\textwidth]{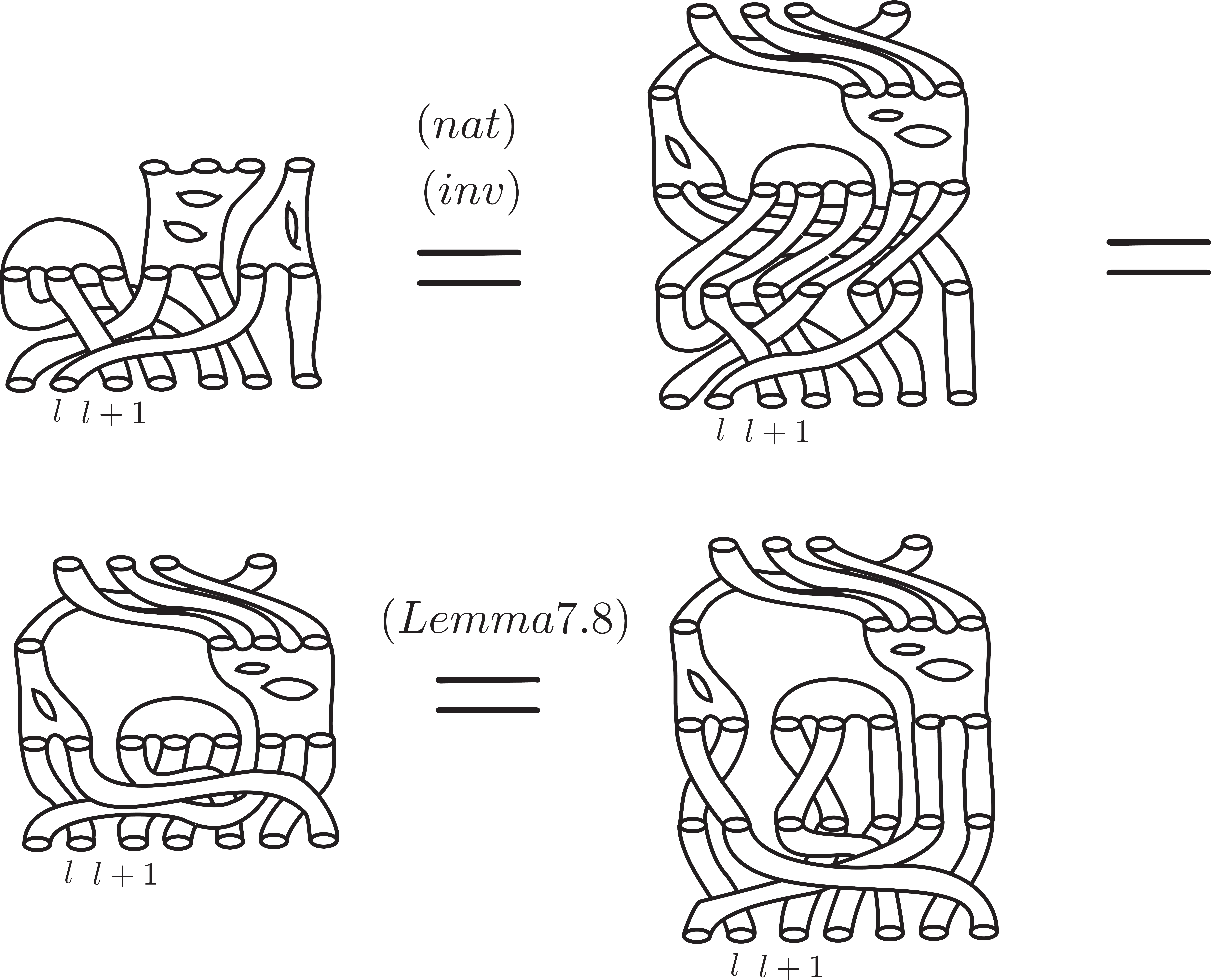}} \caption{Lemma \ref{lem8}}
\label{lemas9}
\end{figure}

If this is not the case, then by the following corollary of
(\ref{nt}) and (\ref{iv})
\[
f_1\otm f_2= \tau_{m_2,m_1}\circ(f_2\otm f_1)\circ\tau_{n_1,n_2},
\]
we may assume, without loss of generality, that there is $j\in
\{1,\ldots,k\}$ such that
\[
\sum_{i=1}^{j-1} p_i\leq \pi^{-1}(l)<\sum_{i=1}^{j} p_i\leq
\pi^{-1}(l+1)< \sum_{i=1}^{j+1} p_i.
\]
If necessary, by Lemma~\ref{lem7} a new $\tau$-term could be added
in between the tail and the center of $f$ in order to obtain a new
tail such that $l,l+1$ are parallel in it.
\end{proof}

The proof of the following lemma is akin to the proof of
Lemma~\ref{lem3}.

\begin{lem}\label{lem9}
For every $\tau$-term $\pi:p\str p$ and every $l\in p-1$ such that
$l,l+1$ are parallel in $\pi$, there is a $\tau$-term
$\pi':p-2\str p-2$ such that for $j=\pi^{-1}(l)$, $\pi$ is equal
to
\[
(\tau_{2,l}\otm \mj_{p-l-2})\circ(\mj_2\otm \pi')\circ
(\tau_{j,2}\otm \mj_{p-j-2}).
\]
\end{lem}

By Lemma~\ref{lem7} and the equations (\ref{nt}), (\ref{fn}) and
(\ref{fb}), we have the following.

\begin{lem}\label{lem10}
For $n\geq 1$,
$(\mj_l\otm\md\otm\mj_{n-l-1})\circ\Lambda_n=\Lambda_{n-1}\circ
H_1$.
\end{lem}

\begin{figure}[h!h!h!]
\centerline{\includegraphics[width=0.8 \textwidth]{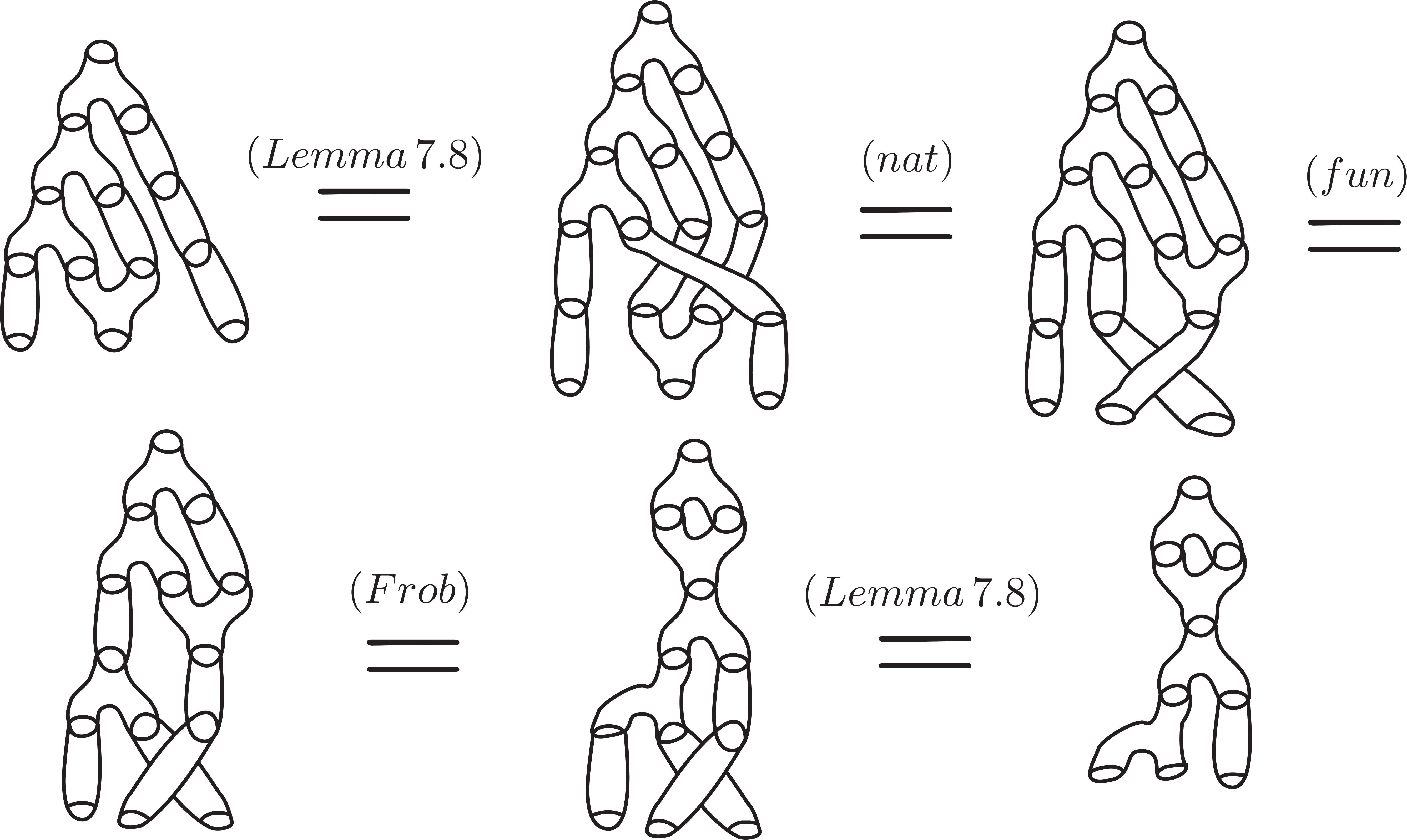}}
\caption{Lemma \ref{lem10}} \label{lemas11}
\end{figure}

By the equations (\ref{fn}) and (\ref{fb}), we have the following.

\begin{lem}\label{lem11}
For $n,m\geq 0$, $(\mj_n\otm\md\otm\mj_m)\circ(\Lambda_n\otm
\Lambda_m)=\Lambda_{n+m}\circ \md$.
\end{lem}

\begin{figure}[h!h!h!]
\centerline{\includegraphics[width=0.75
\textwidth]{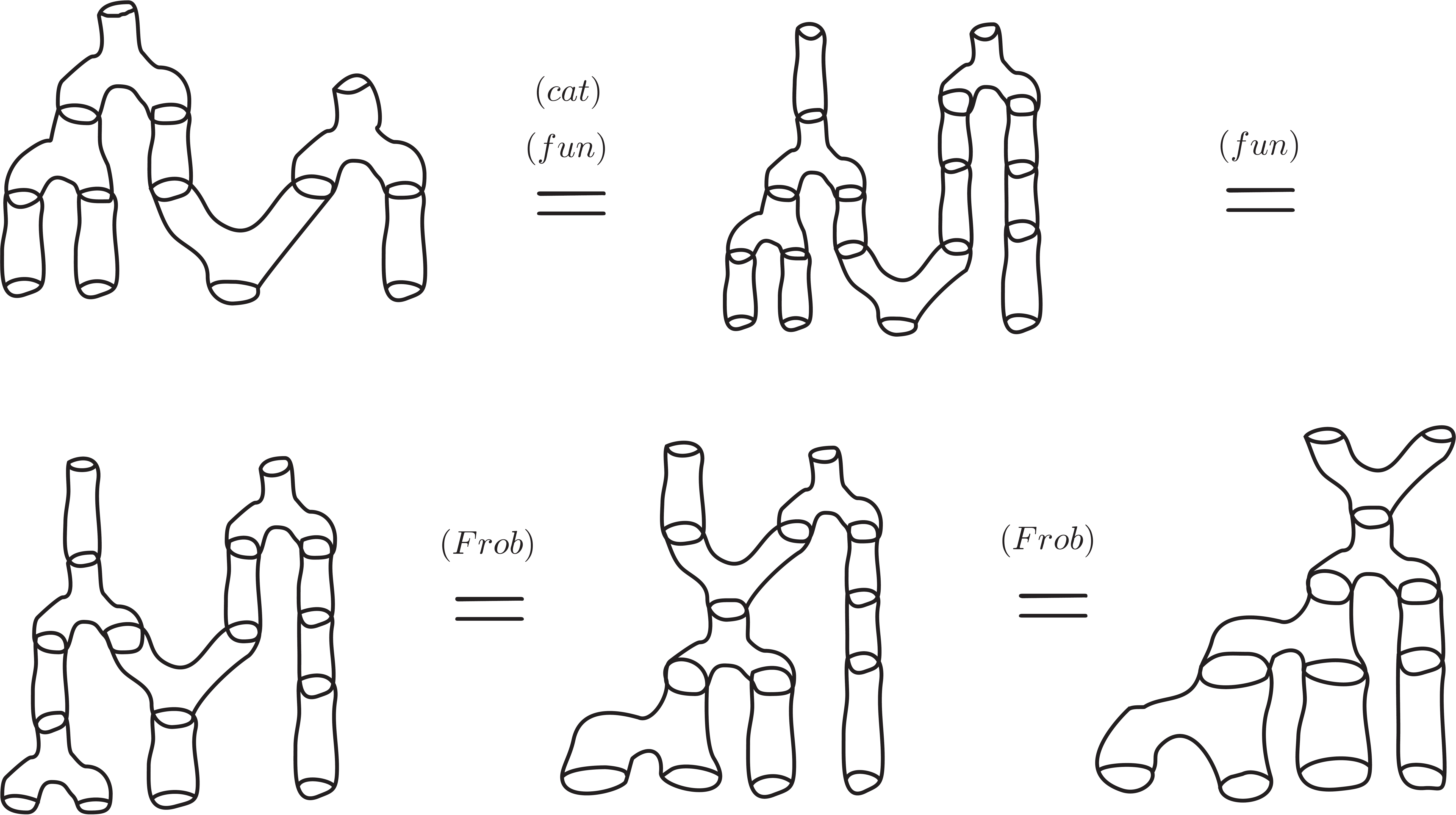}} \caption{Lemma \ref{lem11}}
\label{lemas12}
\end{figure}

By the equations (\ref{fn}), (\ref{as}) and (\ref{fb}), we have
the following.
\begin{lem}\label{lem12}
For $n,m\geq 0$, $\md\circ(H_n\otm H_m)=H_{n+m}\circ \md$.
\end{lem}

\begin{figure}[h!h!h!]
\centerline{\includegraphics[width=0.8
\textwidth]{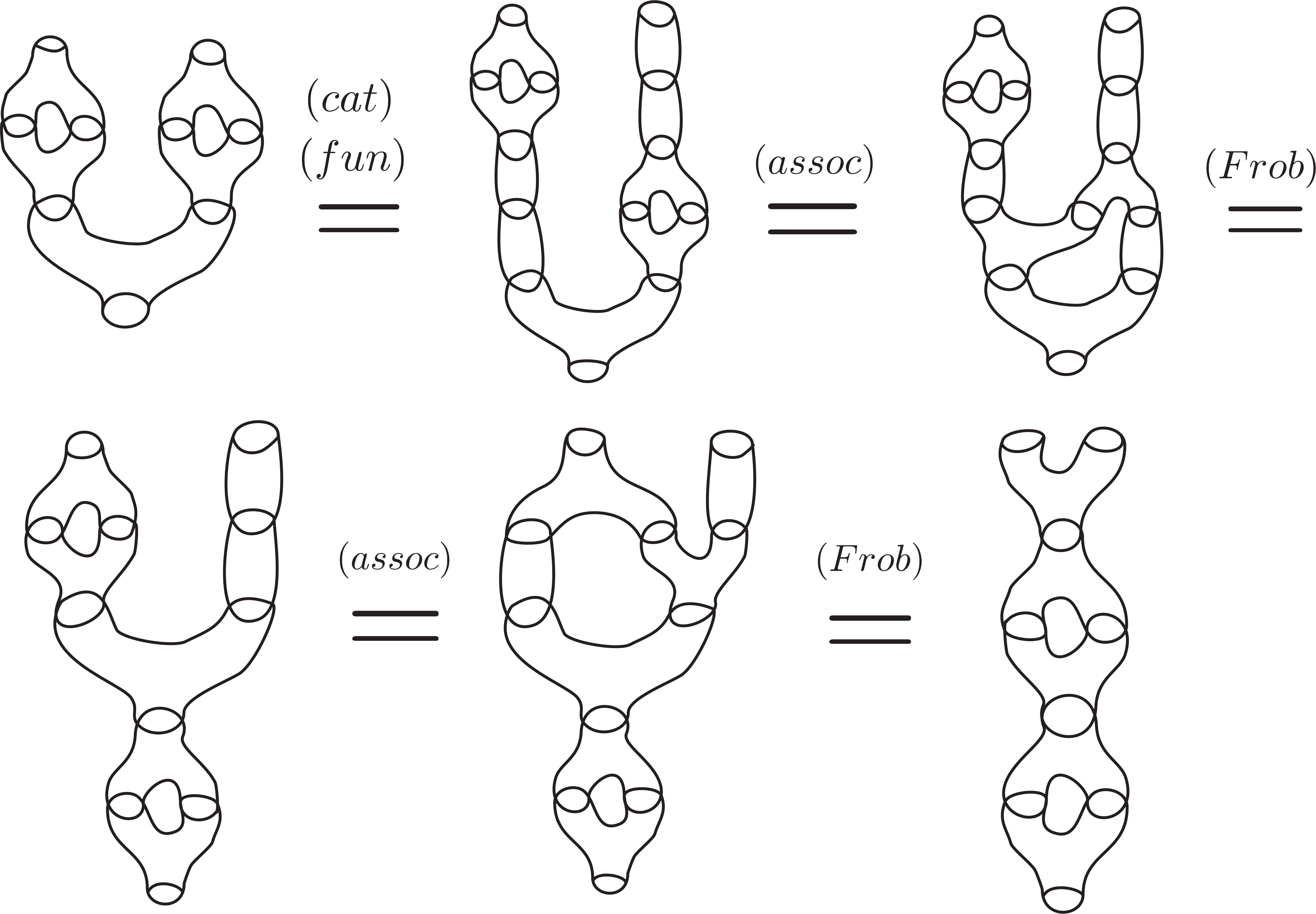}} \caption{Lemma \ref{lem12}}
\label{lemas13}
\end{figure}

By the equations (\ref{fn}), (\ref{as}) or (\ref{un}), we have the
following.

\begin{lem}\label{lem13}
For $n,m\geq -1$, $\md\circ(V_n\otm V_m)=V_{n+m+1}$.
\end{lem}

\begin{figure}[h!h!h!]
\centerline{\includegraphics[width=0.8
\textwidth]{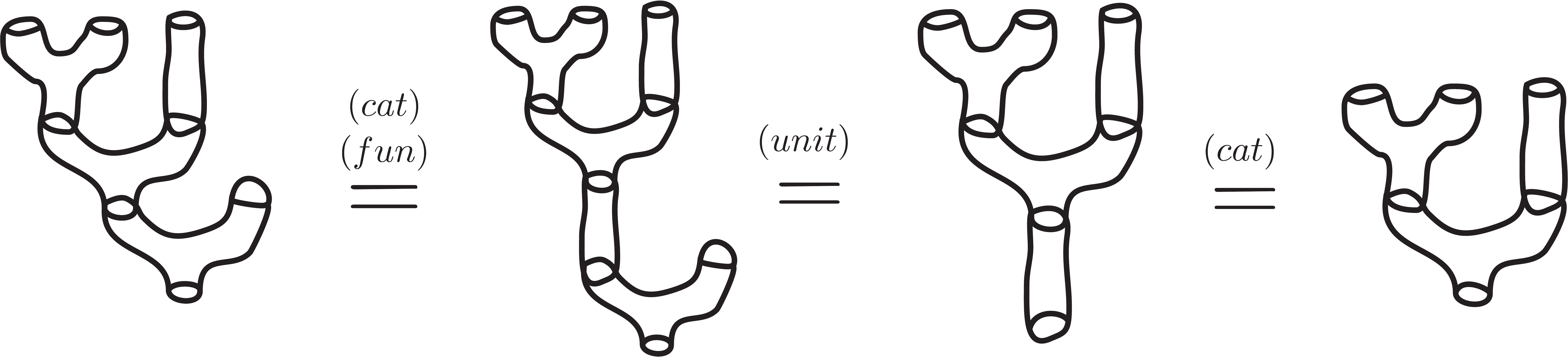}} \caption{Lemma \ref{lem13}, $m=-1$}
\label{lemas14-1}
\end{figure}

\begin{figure}[h!h!h!]
\centerline{\includegraphics[width=0.8
\textwidth]{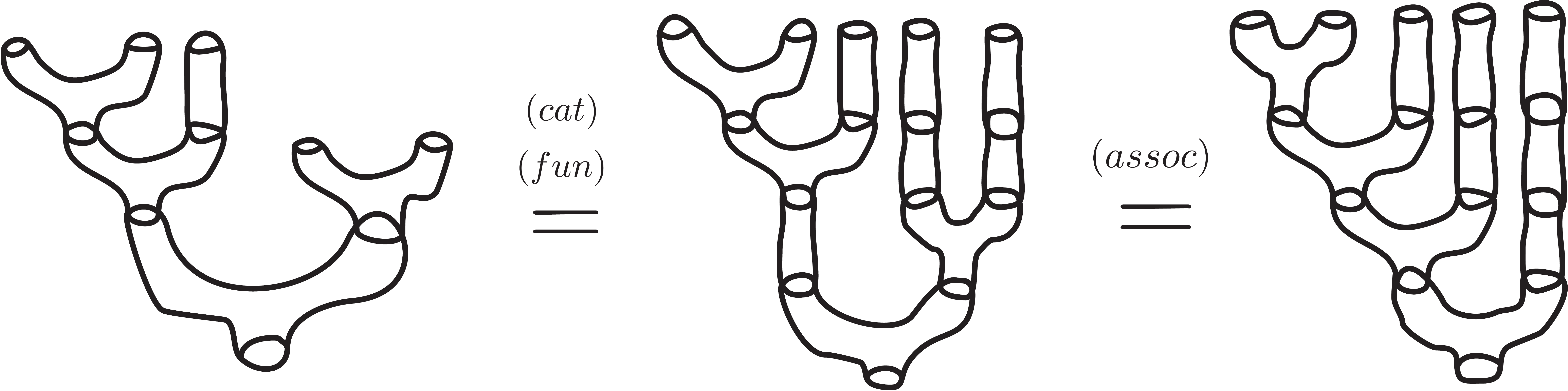}} \caption{Lemma \ref{lem13}}
\label{lemas14}
\end{figure}

\vspace{2ex}

\begin{proof}[Proof of Proposition~\ref{tv1}]
Let $f$ be a term. By Lemma~\ref{lem1}, $f$ is equal to a term of
the form $f_n\circ\ldots\circ f_0$, where every $f_i$ is of the
form $\mj_l\otimes\beta\otimes\mj_r$, for $l,r\geq 0$ and
$\beta\in\{\tau,\ms,\es,\ed,\md\}$. We proceed by induction on
$n\geq 0$. (The indices of identities not important for our
calculations are usually omitted.)

\vspace{2ex}

If $n=0$, then since $\mj_l\otimes\beta\otimes\mj_r$ is special,
we are done.

\vspace{2ex}

If $n>0$, then by the induction hypothesis,
$f_{n-1}\circ\ldots\circ f_0$ is equal to a term of the form
$\pi\circ \bigotimes_{i=1}^k E_{p_i,m_i,n_i} \circ \chi$. We have
the following cases concerning $f_n$.

\vspace{1ex}

If $f_n$ is $\mj_l\otimes\tau\otimes\mj_r$, then $f_n\circ \pi$ is
a $\tau$-term and we are done.

\vspace{1ex}

If $f_n$ is $\mj_l\otimes\ms\otimes\mj_r$, then by
Lemma~\ref{lem3}, we have a $\tau$-term $\pi'$ such that
\begin{tabbing}
\hspace{1.5em}\= $f_n\circ \pi$ \= $=(\mj_l\otm\ms\otm\mj)\circ
(\tau_{1,l}\otm \mj)\circ(\mj_1\otm \pi')\circ (\tau_{j,1}\otm
\mj)$
\\[1ex]
\> \>
$=(\tau_{2,l}\otm\mj)\circ(\ms\otm\mj)\circ(\mj_1\otm\pi')\circ
(\tau_{j,1}\otm \mj)$ \`(\ref{nt})
\\[1ex]
\> \> $=(\tau_{2,l}\otm\mj)\circ(\mj_2\otm\pi')\circ
(\ms\otm\mj)\circ (\tau_{j,1}\otm \mj)$ \`(\ref{ct}), (\ref{fn})
\\[1ex]
\> \> $=(\tau_{2,l}\otm\mj)\circ(\mj_2\otm\pi')\circ
(\tau_{j,2}\otm \mj)\circ (\mj_{j}\otm\ms\otm\mj)$ \`(\ref{nt})
\end{tabbing}
Then for some $u\in k$, by (\ref{fn}),
$(\mj_{j}\otm\ms\otm\mj)\circ \bigotimes_{i=1}^k E_{p_i,m_i,n_i}$
is equal to
\[
(\mj\otm ((\mj\otm\ms\otm\mj)\circ E_{p_u,m_u,n_u})\otm\mj)\circ
(\bigotimes_{i=1}^{u-1} E_{p_i,m_i,n_i}\otm \mj_{n_u}\otm
\bigotimes_{i=u+1}^k E_{p_i,m_i,n_i}),
\]
which is, with the help of Lemma~\ref{lem4} and (\ref{fn}) again,
equal to the new center
\[
\bigotimes_{i=1}^{u-1} E_{p_i,m_i,n_i}\otm E_{p_u+1,m_u,n_u}\otm
\bigotimes_{i=u+1}^k E_{p_i,m_i,n_i}.
\]

\vspace{1ex}

If $f_n$ is $\mj_l\otimes\es\otimes\mj_r$, then we proceed as in
the preceding case, just by relying on Lemma~\ref{lem5} instead of
Lemma~\ref{lem4} in order to obtain the new center
\[
\bigotimes_{i=1}^{u-1} E_{p_i,m_i,n_i}\otm E_{p_u-1,m_u,n_u}\otm
\bigotimes_{i=u+1}^k E_{p_i,m_i,n_i}.
\]

\vspace{1ex}

If $f_n$ is $\mj_l\otimes\ed\otimes\mj_r$, then by relying on
Lemma~\ref{lem6} we have the following
\begin{tabbing}
\hspace{1.5em}\= $f_n\circ \pi\circ \bigotimes_{i=1}^k
E_{p_i,m_i,n_i}$ \= $=(\tau_{1,l}\otm\mj)\circ(\mj_1\otm\pi)\circ
(\ed\otm\mj)\circ \bigotimes_{i=1}^k E_{p_i,m_i,n_i}$
\\[1ex]
\> \> $=(\tau_{1,l}\otm\mj)\circ(\mj_1\otm\pi)\circ
\bigotimes_{i=0}^k E_{p_i,m_i,n_i}$ \`(\ref{fn}),
\end{tabbing}
where $p_0=1$ and $m_0=n_0=0$.

\vspace{1ex}

If $f_n$ is $\mj_l\otimes\md\otimes\mj_r$, then by
Lemmata~\ref{lem8} and~\ref{lem9}, we may assume that the tail
$\pi$ of a special term equal to $f_{n-1}\circ\ldots\circ f_0$ is
of the form
\[
(\tau_{2,l}\otm \mj_{p-l-2})\circ(\mj_2\otm \pi')\circ
(\tau_{j,2}\otm \mj_{p-j-2}).
\]
As above, we obtain
\begin{tabbing}
\hspace{1.5em}\= $f_n\circ
\pi=(\tau_{1,l}\otm\mj)\circ(\mj_1\otm\pi')\circ (\tau_{j,1}\otm
\mj)\circ (\mj_{j}\otm\md\otm\mj)$.
\end{tabbing}

There are two possibilities concerning the term
\[
(\mj_{j}\otm\md\otm\mj)\circ \bigotimes_{i=1}^k E_{p_i,m_i,n_i}.
\]
Either it is equal to
\[
(\mj\otm ((\mj\otm\md\otm\mj)\circ E_{p_u,m_u,n_u})\otm\mj)\circ
(\bigotimes_{i=1}^{u-1} E_{p_i,m_i,n_i}\otm \mj_{n_u}\otm
\bigotimes_{i=u+1}^k E_{p_i,m_i,n_i})
\]
when we apply Lemma~\ref{lem10}, with the help of (\ref{fn}), in
order to obtain
\[
\bigotimes_{i=1}^{u-1} E_{p_i,m_i,n_i}\otm
E_{p_{u}-1,m_u+1,n_u}\otm \bigotimes_{i=u+1}^k E_{p_i,m_i,n_i},
\]
or it is equal to
\begin{tabbing}
\hspace{1.5em}$(\mj\otm ((\mj\otm\md\otm\mj)\circ
(E_{p_u,m_u,n_u}\otm E_{p_{u+1},m_{u+1},n_{u+1}}))\otm\mj)\circ$
\\[1ex]
\` $(\bigotimes_{i=1}^{u-1} E_{p_i,m_i,n_i}\otm
\mj_{n_u+n_{u+1}}\otm \bigotimes_{i=u+2}^k E_{p_i,m_i,n_i})$
\end{tabbing}
when we apply Lemmata~\ref{lem11}, \ref{lem12} and~\ref{lem13} in
order to obtain
\[
\bigotimes_{i=1}^{u-1} E_{p_i,m_i,n_i}\otm
E_{p_{u}+p_{u+1}-1,m_u+m_{u+1},n_u+n_{u+1}} \otm
\bigotimes_{i=u+2}^k E_{p_i,m_i,n_i}
\]
\end{proof}

For $a,b,c,d\geq 0$, and $n_i,q_i,s_i,u_i\geq 1$ consider a
special term of the form
\[
\pi\circ (\bigotimes_{i=1}^a E_{0,m_i,0} \otm \bigotimes_{i=1}^b
E_{0,p_i,n_i} \otm \bigotimes_{i=1}^c E_{q_i,r_i,0}\otm
\bigotimes_{i=1}^d E_{s_i,t_i,u_i})\circ \chi.
\]
If $b\geq 1$, let $\beta^1=0$, and let
$\beta^i=n_1+\ldots+n_{i-1}$, for $i\in\{2,\ldots,b\}$.

If $d\geq 1$, let $\delta^1=n_1+\ldots+n_b$,
$\delta_1=q_1+\ldots+q_c$, and for $i\in\{2,\ldots,d\}$, let
$\delta^i=n_1+\ldots+n_{b}+\ldots+u_1+\ldots+u_{i-1}$ and
$\delta_i=q_1+\ldots+q_{c}+\ldots+s_1+\ldots+s_{i-1}$.

If $c\geq 1$, let $\gamma_1=0$, and for $i\in\{2,\ldots,c\}$, let
$\gamma_i=q_1+\ldots+q_{i-1}$. Such a special term is in
\emph{normal form} when
\[
m_1\leq m_2\leq\ldots\leq m_a,
\]
\[
\chi^{-1}(\beta^1)< \chi^{-1}(\beta^2)<\ldots< \chi^{-1}(\beta^b),
\]
\[
\pi(\gamma_1)< \pi(\gamma_2)<\ldots< \pi(\gamma_c),
\]
\[
\pi(\delta_1)< \pi(\delta_2)<\ldots< \pi(\delta_d),
\]
for every $i\in\{1,\ldots,b\}$
\[
\chi^{-1}(\beta^i)< \chi^{-1}(\beta^i +1)<\ldots<
\chi^{-1}(\beta^i+n_i-1),
\]
for every $i\in\{1,\ldots,d\}$
\[
\chi^{-1}(\delta^i)< \chi^{-1}(\delta^i +1)<\ldots<
\chi^{-1}(\delta^i+u_i-1),
\]
\[
\pi(\delta_i)< \pi(\delta_i +1)<\ldots< \pi(\delta_i+s_i-1),
\]
and finally, for every $i\in\{1,\ldots,c\}$
\[
\pi(\gamma_i)< \pi(\gamma_i +1)<\ldots< \pi(\gamma_i+q_i-1).
\]

By Proposition~\ref{tv1}, Lemma~\ref{lem7} and the equation
\[
f_1\otm f_2= \tau_{m_2,m_1}\circ(f_2\otm f_1)\circ\tau_{n_1,n_2},
\]
which follows from (\ref{nt}) and (\ref{iv}), we can prove the
following.

\begin{thm}\label{thm1}
Every term is equal to a term in normal form.
\end{thm}

\section{Faithfulness of the interpretation}\label{faithfulness}

The aim of this section is to prove the following result.

\begin{thm}\label{completeness}
For every $d\geq 2$, the interpretation of $\mathbf{K}$ in $dCobS$
is faithful.
\end{thm}

For the proof of this theorem, we need some auxiliary notions and
results. Every $d$-cobordism
$K=(M,\Sigma_0,\Sigma_1):\underline{n}\str \underline{m}$,
$\Sigma_0=(\Sigma_0^0,\ldots,\Sigma_0^{n-1})$ and
$\Sigma_1=(\Sigma_1^0,\ldots,\Sigma_1^{m-1})$, induces the
following equivalence relation $\rho_K$ on the set $(n\times
\{0\})\cup (m\times\{1\})$ (cf.\ the relation with the same name
defined in Section~\ref{Brauerian}). For $(i,k)$ and $(j,l)$
elements of $(n\times \{0\})\cup (m\times\{1\})$, we have that
$(i,k)\rho_K (j,l)$ when
\[
\Sigma_k^i\; \mbox{\rm and } \Sigma_l^j \; \mbox{\rm belong to the
same connected component of } M.
\]

From Proposition~\ref{props1}, and since homeomorphisms preserve
connected components, we have the following lemma.

\begin{lem}\label{rhoinvariant}
If two $d$-cobordisms $K=(M,\Sigma_0,\Sigma_1)$ and $L=
(N,\Delta_0,\Delta_1)$ are equivalent, then $\rho_K=\rho_L$.
\end{lem}

The following proposition serves to prove that our categories are
skeletal.

\begin{prop}\label{isomorphism}
If $K:\underline{n}\str \underline{m}$ is an isomorphism, then
$n=m$.
\end{prop}

\begin{proof}
We prove that every equivalence class of $\rho_K$ has two
elements, one with the second component 0 and the other with the
second component 1, from which the proposition follows. Let
$L:\underline{m}\str \underline{n}$ be the inverse of $K$.

Suppose that an equivalence class of $\rho_K$ is a singleton
$\{(i,0)\}$. Then $\{(i,0)\}$ is an equivalence class of
$\rho_{L\circ K}$, which is impossible by
Lemma~\ref{rhoinvariant}, since $L\circ K$ is equivalent to the
identity $d$-cobordism.

Suppose that for $i\neq j$, an equivalence class of $\rho_K$
contains both $(i,0)$ and $(j,0)$. Then an equivalence class of
$\rho_{L\circ K}$ contains both $(i,0)$ and $(j,0)$, which is
again impossible by Lemma~\ref{rhoinvariant}.

We proceed analogously, by relying on $\rho_{K\circ L}$, in cases
when an equivalence class of $\rho_K$ is a singleton $\{(i,1)\}$
or when for $i\neq j$, an equivalence class of $\rho_K$ contains
both $(i,1)$ and $(j,1)$.
\end{proof}

\begin{cor}\label{skeletal}
The categories $dCobS$, for $d\geq 2$, and $\mathbf{K}$ are
skeletal.
\end{cor}
That $1CobS$ is also skeletal is proved in
Section~\ref{Brauerian}. The following implication has a trivial
converse.

\begin{lem}\label{holes}
If $\underline{E_{0,n,0}}\sim \underline{E_{0,m,0}}$, then $n=m$.
\end{lem}

\begin{proof}
The $d$-manifolds underlying the cobordisms
$\underline{E_{0,n,0}}$ and $\underline{E_{0,m,0}}$ are closed,
and these $d$-cobordisms can be identified with the underlying
manifolds. Moreover, $\underline{E_{0,n,0}}\sim
\underline{E_{0,m,0}}$ means that these manifolds are
homeomorphic. In the case when $d=2$, we have that $n$ is the
genus of $\underline{E_{0,n,0}}$, and when $d\geq 3$, we have that
$\underline{E_{0,1,0}}$ is homeomorphic to $S^{d-1}\times S^1$,
which with the help of Van Kampen's Theorem asserts that the
fundamental group of $\underline{E_{0,n,0}}$ is the free group
with $n$ generators.
\end{proof}

In the sequel, we assume that $f$ and $f'$ are two normal forms
\[
\pi\circ (\bigotimes_{i=1}^a E_{0,m_i,0} \otm \bigotimes_{i=1}^b
E_{0,p_i,n_i} \otm \bigotimes_{i=1}^c E_{q_i,r_i,0}\otm
\bigotimes_{i=1}^d E_{s_i,t_i,u_i})\circ \chi
\]
and
\[
\pi'\circ (\bigotimes_{i=1}^{a'} E_{0,m'_i,0} \otm
\bigotimes_{i=1}^{b'} E_{0,p'_i,n'_i} \otm \bigotimes_{i=1}^{c'}
E_{q'_i,r'_i,0}\otm \bigotimes_{i=1}^{d'} E_{s'_i,t'_i,u'_i})\circ
\chi'.
\]

\begin{prop}\label{corholes}
If $\underline{f}\sim \underline{f'}$, then $a=a'$ and $m_i=m'_i$
for every $1\leq i\leq a$.
\end{prop}

\begin{proof}
Since every homeomorphism justifying $\underline{f}\sim
\underline{f'}$ maps the closed components of $\underline{f}$ to
the closed components of $\underline{f'}$, there must be a
bijection from $\{1,\ldots,a\}$ to $\{1,\ldots,a'\}$ such that
$\underline{E_{0,m_i,0}}\sim \underline{E_{0,m'_j,0}}$, for $j$
corresponding to $i$ by this bijection. Hence, we have $a=a'$, and
by Lemma~\ref{holes}, since the sequences $(m_i)$ and $(m'_i)$ are
increasing, we conclude that $m_i=m'_i$ for every $1\leq i\leq a$.
\end{proof}

The following proposition has Theorem~\ref{completeness} as an
immediate corollary.

\begin{prop}
If $\underline{f}\sim \underline{f'}$, then $f$ and $f'$ are
identical.
\end{prop}

\begin{proof}
By Proposition~\ref{props1} we have that $f$ and $f'$ are of the
same type $n\str m$. We proceed by induction on $n+m$. If $n+m=0$,
then we apply Proposition~\ref{corholes}.

If $n+m>0$, let $\rho$ be the equivalence relation corresponding,
by Lemma~\ref{rhoinvariant}, both to $\underline{f}$ and
$\underline{f'}$. Suppose that $b>0$, hence $E_{0,p_1,n_1}$
appears in $f$. The relation $\rho$ guaranties that $b'>0$. Let
$X=\chi^{-1}[\{0,\ldots, n_1-1\}]$. The set $X\times\{0\}$ is an
equivalence class of $\rho$, namely the equivalence class of
$(\chi^{-1}(0),0)$. Our normal form and the relation $\rho$
guarantee that $n'_1=n_1$, and that $\chi$ and $\chi'$ coincide
on~$X$.

Let $g$ be the term $g_0\otimes\ldots\otimes g_{n-1}$, where
\[
g_i=\left\{ \begin{array}{ll} \mj_1, & i\not\in X, \\
\ed, & i\in X.\end{array}\right.
\]
By relying on the equation (\ref{nt}), $f\circ g$ is equal to the
normal form $f_1$
\[
\pi\circ (A \otm \bigotimes_{i=2}^b E_{0,p_i,n_i} \otm C \otm
D)\circ \chi_1,
\]
where $A$ is of the form
\[
\bigotimes_{i=1}^{k} E_{0,m_i,0} \otm E_{0,p_1,0}\otm
\bigotimes_{i=k+1}^{a} E_{0,m_i,0},
\]
while $C$ is $\bigotimes_{i=1}^c E_{q_i,r_i,0}$, and $D$ is
$\bigotimes_{i=1}^d E_{s_i,t_i,u_i}$. Analogously, we conclude
that $f'\circ g$ is equal to the normal form $f'_1$
\[
\pi'\circ (A' \otm \bigotimes_{i=2}^{b'} E_{0,p'_i,n'_i} \otm C'
\otm D')\circ \chi'_1,
\]
with the abbreviations $A'$, $C'$ and $D'$ as above.

From $\underline{f\circ g}\sim \underline{f'\circ g}$, since the
interpretation is a functor, it follows that $\underline{f_1}\sim
\underline{f'_1}$. By the induction hypothesis $f_1$ and $f'_1$
are identical. We have that $\bigotimes_{i=1}^a E_{0,m_i,0}$ and
$\bigotimes_{i=1}^{a'} E_{0,m'_i,0}$ are identical, by
Proposition~\ref{corholes}, which together with the fact that $A$
and $A'$ are identical delivers that $p_1=p'_1$. It remains only
to prove that the permutations $\chi$ and $\chi'$ are equal, which
follows from the fact that $\chi_1$ and $\chi'_1$ are equal and
that $\chi$ and $\chi'$ coincide on $X$.

We proceed analogously in all the other situations ($b=0$ and
$c>0$, or $b=c=0$ and $d>0$).
\end{proof}

Since the interpretation of $\mathbf{K}$ in $dCobS$ is one-one on
objects, it is an embedding. The following corollary asserting
that $2CobS$ is a PROP having $\underline{1}$ as the universal
commutative Frobenius object is already given in
\cite[Theorem~3.6.19]{K03}.

\begin{cor}\label{universal}
The category $\mathbf{K}$ is isomorphic to $2CobS$.
\end{cor}

\begin{proof}
From the classification theorem for 2-manifolds (see for example
\cite[VI.40]{ST80}) it follows that, in the case $d=2$, the
interpretation is full.
\end{proof}

However, for $d>2$, the interpretation is not full, and hence not
an isomorphism.

\section{Appendix}

\subsection{Topological manifolds, orientation and
gluing}\label{orientation}

For $n\geq 0$, an $n$-\emph{dimensional manifold} $M$ is a second
countable Hausdorff space that is locally Euclidean of dimension
$n$. This means that the topology of $M$ admits a countable basis,
that there are disjoint neighborhoods of every pair of distinct
points in $M$, and that every point in $M$ has a neighborhood
homeomorphic to an open subset of $\mR^n$. A \emph{chart} of $M$
is a homeomorphism $\varphi:U\str U'$, where $U\subseteq M$ and
$U'\subseteq\mR^n$ are open. An \emph{atlas} of $M$ is a
collection of its charts $\{\varphi_i:U_i\str U'_i\mid i\in I\}$
such that $\bigcup\{U_i\mid i\in I\}=M$.

For $n\geq 1$, an $n$-dimensional \emph{manifold with boundary},
shortly $\partial$-\emph{manifold}, is a second countable
Hausdorff space in which every point has a neighborhood
homeomorphic to an open subset of the halfspace
\[
\pi^+_n=\{(x_1,\ldots,x_n)\in\mR^n\mid x_n\geq 0\}.
\]
A \emph{chart} of an $n$-dimensional $\partial$-manifold $M$ is a
homeomorphism $\varphi:U\str U'$, where $U\subseteq M$ and
$U'\subseteq\pi^+_n$ are open. An \emph{atlas} of $M$ is again a
collection of its charts whose domains cover $M$.

A \emph{boundary point} of $M$ is a point mapped to a point in the
hyperplane
\[
\pi_n=\{(x_1,\ldots,x_n)\in\mR^n\mid x_n= 0\}
\]
by some chart, otherwise, it is an \emph{interior point} of $M$.
The set of boundary points of $M$ is its \emph{boundary} $\partial
M$, which is an $(n-1)$-dimensional manifold, and the set of
interior points of $M$ is its \emph{interior} $\Int M$, which is
an $n$-dimensional manifold. The \emph{interior} $\Int U$ of an
open subset $U$ of $M$ is $U-\partial M$. Every $n$-dimensional
manifold, for $n\geq 1$, is an $n$-dimensional
$\partial$-manifold, with the empty boundary.

\vspace{1ex}

A homeomorphism $f:U\str V$ for open $U,V\subseteq \mR^n$, $n\geq
1$, is \emph{orientation preserving} when for every $x\in U$ the
following isomorphism of homology groups with coefficients in
$\mZ$ is the identity
\[
H_n(\mR^n,\mR^n\!-\{0\})\stackrel{\cong\;}{\str}
H_n(U,U-\{x\})\stackrel{f_\ast\;}{\str}
H_n(V,V-\{f(x)\})\stackrel{\cong\;}{\str}
H_n(\mR^n,\mR^n\!-\{0\}).
\]
Here, the first isomorphism is the composition
\[
H_n(\mR^n,\mR^n\!-\{0\})\stackrel{(t_x)_\ast}{\longrightarrow}
H_n(\mR^n,\mR^n\!-\{x\})\stackrel{\rm excision}{\longrightarrow}
H_n(U,U-\{x\}),
\]
where $t_x:\mR^n\str \mR^n$ is the translation by $x$, and the
last isomorphism is defined analogously.

\begin{lem}\label{lema trans}
Let $\{W_i\mid i\in I\}$ be an open cover of an open subset $U$ of
$\mR^n$. A homeomorphism $f:U\str V$, for $V$ an open subset of
$\mR^n$, is orientation preserving iff for every $i\in I$, the
restriction of $f$ to $W_i$ is orientation preserving.
\end{lem}

An atlas $\{\varphi_i:U_i\str U'_i\mid i\in I\}$ of an
$n$-dimensional manifold, $n\geq 1$, is \emph{oriented} when for
every $i,j\in I$, the homeomorphism
\[
\varphi_i\circ\varphi_j^{-1}:\varphi_j[U_i\cap
U_j]\str\varphi_i[U_i\cap U_j]
\]
is orientation preserving. A manifold possessing such an atlas is
\emph{orientable}. An oriented atlas is \emph{maximal} when it
cannot be enlarged to an oriented atlas of the manifold by adding
another chart.

Two oriented atlases $\{\varphi_i:U_i\str U'_i\mid i\in I\}$ and
$\{\psi_j:V_j\str V'_j\mid j\in J\}$ of the same manifold are
\emph{equivalent} when, for every $i\in I$ and every $j\in J$, the
homeomorphism
\[
\varphi_i\circ\psi_j^{-1}:\psi_j[U_i\cap V_j]\str\varphi_i[U_i\cap
V_j]
\]
is orientation preserving (cf.\ \cite[Definition~21.11]{T08}).

\begin{prop}\label{union}
If two oriented atlases of a manifold are equivalent, then their
union is an oriented atlas of this manifold.
\end{prop}

With the help of Lemma~\ref{lema trans} for transitivity, we can
prove the following.

\begin{prop}\label{equivalence}
The above relation is an equivalence relation on the set of
oriented atlases of an orientable manifold.
\end{prop}

If an orientable manifold is connected, then this equivalence
relation has exactly two classes. As a corollary of Propositions
\ref{union} and \ref{equivalence}, we have the following.

\begin{prop}\label{maximal}
Every oriented atlas could be enlarged to a unique maximal
oriented atlas.
\end{prop}

An \emph{orientation} of a $0$-dimensional manifold $M$ is a
function $\varepsilon:M\str\{-1,1\}$. For $n\geq 1$, an
\emph{orientation} of an orientable $n$-dimensional manifold $M$
is a choice of its maximal oriented atlas $\mathcal{O}_M$. The
orientation opposite to $\mathcal{O}_M$ is obtained by composing
every chart in it by a reflection of $\mR^n$, for example with
respect to $\pi_n$.

The \emph{orientation of the product} of two oriented manifolds
$M$ and $N$ is given by the maximal oriented atlas containing the
products of charts in $\mathcal{O}_M$ with charts in
$\mathcal{O}_N$. A homeomorphism $f$ between two oriented
$n$-dimensional manifolds $M$ and $N$ is \emph{orientation
preserving} when for every chart $\varphi:U\str U'$ of $M$, for
$g$ being the restriction of $f^{-1}$ to $f[U]$, we have that
\[
\varphi\in \mathcal{O}_M\quad\mbox{\rm iff}\quad \varphi\circ g\in
\mathcal{O}_N.
\]
An embedding of an $n$-dimensional manifold into an
$n$-dimensional manifold is \emph{orientation preserving} when its
restriction to the image is such. An \emph{orientation reversing}
homeomorphism (embedding) from $M$ to $N$ is an orientation
preserving homeomorphism (embedding) from $M$ to $N$ with the
opposite orientation.

An $n$-dimensional $\partial$-manifold, for $n\geq 1$, is
orientable when its interior is orientable and an orientation of
the interior is an orientation of the $\partial$-manifold. We
denote the orientation of an oriented $\partial$-manifold $M$
again by $\mathcal{O}_M$. We say that an oriented $n$-dimensional
$\partial$-manifold $M\subseteq\mR^n$ is \emph{oriented by the
identity} when its orientation contains the charts $\mj_U:U\str U$
for every open $U\subseteq \Int M$.

The orientation of an oriented $\partial$-manifold \emph{induces
the orientation} of its boundary in the following way. For an
oriented $1$-dimensional $\partial$-manifold $M$ and $x\in\partial
M$, we orient $x$ by $\varepsilon(x)=1$, when for a neighborhood
$U$ of $x$ in $M$ there is a chart $\varphi:U\str U'$,
$U'\subseteq \{y\in\mR\mid y\geq 0\}$, such that its restriction
to $\Int U$ is in $\mathcal{O}_M$. Otherwise, we orient $x$ by
$\varepsilon(x)=-1$. For example, if $I=[0,1]$ is oriented by the
identity, then $\varepsilon(0)=1$ and $\varepsilon(1)=-1$. (Note
that this is opposite to the orientation given in \cite{K13} but
it is consistent with the orientation given in \cite{K03}.)

An orientation of the sphere $S^0$ is taken to be induced from an
orientation of the interval $[-1,1]$. Hence, in every orientation
of $S^0$, one point is positive and the other is negative.

For $n\geq 2$, an oriented $n$-dimensional $\partial$-manifold $M$
induces the orientation of $\partial M$ given by the maximal
oriented atlas containing the restriction of $\varphi$ to
$\partial U$ for every chart $\varphi:U\str U'$, $U'\subseteq
\pi_n^+$, whose restriction to $\Int U$ belongs to
$\mathcal{O}_M$. For example, if $\pi_n^+$ is oriented by the
identity, then its boundary $\pi_n$ is oriented by the identity.
If $\pi_n^-=\{(x_1,\ldots,x_n)\in\mR^n\mid x_n\leq 0\}$ is
oriented by the identity, then its boundary $\pi_n$ is oriented by
the maximal oriented atlas containing the restriction of the
reflection $g:\pi_n\str\pi_n$, given by
\[
g(x_1,x_2,\ldots,x_{n-1},0)=(-x_1,x_2,\ldots,x_{n-1},0),
\]
to every open $U\subseteq \pi_n$, i.e.\ it has the opposite
orientation from the one in the previous example.

Let $\Sigma_M$ be a collection of connected components of the
boundary of an $n$-dimensional $\partial$-manifold $M$. An
embedding of an oriented $(n-1)$-manifold into $M$, whose image is
$\Sigma_M$, is \emph{orientation preserving} (\emph{reversing})
when its restriction to the image, with respect to the induced
orientation of $\Sigma_M$, is such.

\vspace{1ex}

We discuss now pushouts in the category of topological spaces, and
in particular the case involving $\partial$-manifolds and oriented
$\partial$-manifolds. For topological spaces $X$, $Y$ and $Z$ and
continuous functions $f:Z\str X$ and $g:Z\str Y$, let $\asymp$ be
the smallest equivalence relation on the disjoint union
\[
X+Y=(X\times\{0\})\cup(Y\times\{1\})
\]
such that for every $z\in Z$ we have that
$(f(z),0)\asymp(g(z),1)$.

For functions $i:X\str (X+Y)/\!\asymp$ and $j:Y\str
(X+Y)/\!\asymp$ defined by $i(x)=[(x,0)]_\asymp$ and
$j(y)=[(y,1)]_\asymp$, let the topological space $X+_{f,g}Y$ be
given by the set $(X+Y)/\!\asymp$ with the topology
\[
\mathcal{T}=\{U\subseteq (X+Y)/\!\asymp\;\;\mid
i^{-1}[U]\;\mbox{\rm is open in }X\;\mbox{\rm and }
j^{-1}[U]\;\mbox{\rm is open in }Y\}.
\]
This is a \emph{pushout} in the category of topological spaces,
i.e.\ we have the commutative diagram
\begin{center}
\begin{picture}(120,50)(0,-5)

\put(0,40){\makebox(0,0){$Z$}} \put(0,0){\makebox(0,0){$X$}}
\put(75,40){\makebox(0,0){$Y$}}
\put(80,0){\makebox(0,0){$X+_{f,g}Y$}}

\put(35,47){\makebox(0,0){$g$}} \put(35,-7){\makebox(0,0){$i$}}
\put(-8,20){\makebox(0,0){$f$}} \put(83,20){\makebox(0,0){$j$}}

\put(15,40){\vector(1,0){40}} \put(15,0){\vector(1,0){40}}
\put(0,30){\vector(0,-1){20}} \put(75,30){\vector(0,-1){20}}
\end{picture}
\end{center}
with the following universal property. For every pair of
continuous functions $i':X\str A$ and $j':Y\str A$ such that
$i'\circ f=j'\circ g$, there is a unique continuous function
$h:X+_{f,g}Y\str A$ such that $h\circ i=i'$ and $h\circ j=j'$.

Let $M$ and $N$ be two $n$-dimensional $\partial$-manifolds and
let $\Sigma_M$ and $\Sigma_N$ be collections of connected
components of $\partial M$ and $\partial N$ respectively, such
that $\Sigma_M$ and $\Sigma_N$ are both homeomorphic to an
$(n-1)$-dimensional manifold $\Sigma$. Let $f:\Sigma\str M$ and
$g:\Sigma\str N$ be two embeddings whose images are $\Sigma_M$ and
$\Sigma_N$ respectively.

\begin{prop}
The space $M+_{f,g}N$ is an $n$-dimensional $\partial$-manifold.
\end{prop}

\begin{proof}
Note that for an $n$-dimensional $\partial$-manifold $M$ we have
that if $K$ is a connected component of $\partial M$, then $M-K$
is an $n$-dimensional $\partial$-manifold whose boundary is
$\partial M-K$. Then we rely on \cite[Chapter~VIII,
Proposition~1.11]{D72}.
\end{proof}

In the case when $M$ and $N$ are two orientable $n$-dimensional
$\partial$-manifolds and $\Sigma_M$, $\Sigma_N$ and $\Sigma$ are
as above, let $f:\Sigma\str M$ be an orientation preserving
embedding whose image is $\Sigma_M$, and let $g:\Sigma\str N$ be
an orientation reversing embedding whose image is $\Sigma_N$. Then
the $n$-dimensional $\partial$-manifold $M+_{f,g}N$ is orientable.

For charts $\varphi:U\str U'$ and $\psi:V\str V'$ of $M$ and $N$
respectively, such that there is $\Gamma\subseteq\Sigma$, possibly
empty, with $\partial U=f[\Gamma]$ and $\partial V=g[\Gamma]$, let
$\varphi+_{f,g}\psi$ be the homeomorphism from $U+_{f,g} V$ to
$U'+_{\varphi\circ f,\psi\circ g}V'$, where by $f$ and $g$ we mean
their restrictions to $\Gamma$. This homeomorphism exists by the
universal property of pushout. We define the orientation of
$M+_{f,g}N$ to be the maximal oriented atlas containing
$\varphi+_{f,g}\psi$ for every pair of charts $\varphi$ and $\psi$
as above such that the restriction of $\varphi$ to $\Int U$ is in
$\mathcal{O}_M$ and the restriction of $\psi$ to $\Int V$ is in
$\mathcal{O}_N$. In this way the restrictions to the interiors of
the embeddings $i:M\str M+_{f,g}N$ and $j:N\str M+_{f,g}N$ are
orientation preserving.

Locally, the situation is completely illustrated by the following
example. For $n\geq 2$, let $\pi_n^+$ and $\pi_n$ be oriented by
the identity. For $f,g:\pi_n\str \pi_n^+$ being orientation
preserving, respectively orientation reversing, embeddings, with
$\pi_n$ as the image, consider the $n$-dimensional manifold
$\pi_n^+ +_{f,g} \pi_n^+$. Without loss of generality, we may
assume that $f$ is the inclusion and that $g$ is the reflection
\[
g(x_1,x_2,\ldots,x_{n-1},0)=(-x_1,x_2,\ldots,x_{n-1},0).
\]

Let $\mathbf{g}:\pi_n^+\str\pi_n^-$ be the composition of two
reflections of $\mR^n$---one with respect to the hyperplane
$\pi_1=\{(x_1,\ldots,x_n)\in\mR^n\mid x_1= 0\}$ and the other with
respect to the hyperplane $\pi_n$. Note that $\mathbf{g}$ is
orientation preserving and its restriction to $\pi_n$ is the
reflection $g:\pi_n\str\pi_n$ from above. Hence, $\mathbf{g}$
reverses the orientation of the boundary. However, the composition
$\mathbf{g}\circ g:\pi_n\str \pi_n^-$ is the inclusion.

Now we have the following commutative diagram
\begin{center}
\begin{picture}(200,110)(0,-5)

\put(0,90){\makebox(0,0){$\pi_n^+$}}
\put(0,10){\makebox(0,0){$\pi_n^+$}}
\put(200,90){\makebox(0,0){$\pi_n^+$}}
\put(200,10){\makebox(0,0){$\pi_n^-$}}
\put(100,90){\makebox(0,0){$\pi_n$}}
\put(105,50){\makebox(0,0){$\pi_n^+ +_{f,g} \pi_n^+$}}
\put(103,10){\makebox(0,0){$\mR^n$}}

\put(50,97){\makebox(0,0){$\supseteq$}}
\put(150,17){\makebox(0,0){$\supseteq$}}
\put(50,17){\makebox(0,0){$\subseteq$}}
\put(150,97){\makebox(0,0){$g$}} \put(-7,50){\makebox(0,0){$\mj$}}
\put(207,50){\makebox(0,0){$\mathbf{g}$}}
\put(45,57){\makebox(0,0){$i$}} \put(155,57){\makebox(0,0){$j$}}
\put(107,30){\makebox(0,0){$h$}}

\put(85,90){\vector(-1,0){70}} \put(115,90){\vector(1,0){70}}
\put(0,75){\vector(0,-1){50}} \put(200,75){\vector(0,-1){50}}
\multiput(100,40)(0,-2){10}{\circle*{.1}}
\put(100,22){\vector(0,-1){2}} \put(15,10){\vector(1,0){70}}
\put(185,10){\vector(-1,0){70}} \put(15,75){\vector(3,-1){60}}
\put(185,75){\vector(-3,-1){55}}

\end{picture}
\end{center}
which by the universal property of pushout leads to the
homeomorphism \linebreak ${h:\pi_n^+ +_{f,g} \pi_n^+\str \mR^n}$.
This homeomorphism is orientation preserving when $\mR^n$ is
oriented by the identity.

\subsection{Some topological remarks}\label{topological}

The classical results formulated in this section are used in
Section~\ref{why}. The following theorem is proved for $n=2$ by
Rad\' o, \cite{R24}, for $n=3$ by Moise, \cite{M52}, for $n=4$ by
Quinn, \cite{Q82}, for $n\geq 5$ by Kirby, \cite{K69}, and it is
trivial for $n=1$.

\begin{thm}[Annulus conjecture, $AC_n$]
Let $f,g:S^{n-1}\str \mathbf{R}^n$ be disjoint, locally flat
embeddings with $f[S^{n-1}]$ inside the bounded component of
$\mathbf{R}^n-g[S^{n-1}]$. Then the closed region bounded by
$f[S^{n-1}]$ and $g[S^{n-1}]$ is homeomorphic to $S^{n-1}\times
I$.
\end{thm}

A homeomorphism from $\mathbf{R}^n$ to $\mathbf{R}^n$ or from
$S^n$ to $S^n$ is called \emph{stable}, when it is equal to a
finite composition of homeomorphisms each of which is the identity
on some non-empty open set. Brown and Gluck, \cite{BG64}, proved
that Annulus conjecture is equivalent to the following statement,
which is hence a theorem.

\begin{thm}[Stable homeomorphism conjecture]\label{SHC}
Any orientation preserving homeomorphism of $\mathbf{R}^n$ is
stable.
\end{thm}

Two homeomorphisms $f,g:X\str Y$ are \emph{isotopic} when there is
a homotopy $\Phi:X\times I\str Y$ from $f$ to $g$ such that every
$\Phi_t:X\str Y$ is a homeomorphism. Such a homotopy is called
\emph{isotopy}.

\begin{thm}[Alexander]\label{Alexander} Every homeomorphism from $\mathbf{R}^n$ to
$\mathbf{R}^n$, or from $S^n$ to $S^n$, whose restriction to some
non-empty open set is the identity, is isotopic to the identity.
\end{thm}



\begin{lem}\label{compiso}
If $\Phi_t:X\str X$ is an isotopy from $f$ to $g$ and
$\Gamma_t:X\str X$ is an isotopy from $u$ to $v$, then
$\Gamma_t\circ \Phi_t$ is an isotopy from $u\circ f$ to $v\circ
g$.
\end{lem}

\begin{prop}\label{prop5}
Every orientation preserving homeomorphism $f:S^n\str S^n$ is
isotopic to the identity.
\end{prop}

\begin{proof} Let $p\in S^n$ and let
$g:S^n\str S^n$ be a homeomorphism whose restriction to some
non-empty open set is the identity, and such that $g(f(p))=p$. (It
is not difficult to construct such a $g$). For $h=g\circ f$ we
have that its restriction to $S^n-\{p\}$, which is homeomorphic to
$\mathbf{R}^n$, is a homeomorphism from $S^n-\{p\}$ to
$S^n-\{p\}$.

By Theorem~\ref{SHC}, we have that this restriction is equal to a
composition of homeomorphisms $h_k\circ\ldots\circ h_1$ such that
every $h_i$ restricted to some non-empty open set is the identity.
If we define $h_i(p)=p$, then every $h_i:S^n\str S^n$ is a
homeomorphism and $f=g^{-1}\circ h_k\circ\ldots\circ h_1$. Hence,
$f$ is stable. By Theorem~\ref{Alexander}, and
Lemma~\ref{compiso}, $f$ is isotopic to the identity.
\end{proof}

\begin{thm}[Invariance of Domain Theorem]\label{invariance}
If $M$ and $N$ are topological $n$-manifolds without boundaries
and $f:M\str N$ is a continuous 1-1 map, then $f$ is an open map.
\end{thm}

\begin{lem}[Pasting Lemma]\label{pasting}
For $X,Y$ both closed or both open subsets of $A=X\cup Y$, if for
$f:A\str B$ both its restrictions to $X$ and $Y$ are continuous,
then $f$ is continuous too.
\end{lem}

\begin{prop}\label{prop8}
If $\Phi_t:S^n\str S^n$ is an isotopy from the identity to $f$,
then $F:S^n\times I\str S^n\times I$ defined by
$F(x,t)=(\Phi_t(x),t)$ is a homeomorphism.
\end{prop}

\begin{proof} We have that $F$ is
continuous and that $F^{-1}$ defined by
$F^{-1}(x,t)=(\Phi^{-1}_t(x),t)$ is its inverse. It remains to
prove that $F^{-1}$ is continuous.

Let $G:S^n\times \mathbf{R}\str S^n\times \mathbf{R}$ be defined
by
\[
G(x,t)=\left\{
\begin{array}{ll}
(x,t), & (x,t)\in S^n\times (-\infty,0],
\\
F(x,t), & (x,t)\in S^n\times [0,1],
\\
(f(x),t), & (x,t)\in S^n\times [1,+\infty).
\end{array}
\right.
\]
We have that $G$ is 1-1 and by Lemma~\ref{pasting} it is
continuous. The $(n+1)$-manifold $S^n\times \mathbf{R}$ is without
boundary, and by Theorem~\ref{invariance}, we have that $G$ is
open. Hence, $F$ is open, which means that $F^{-1}$ is continuous.
\end{proof}

As a corollary of Propositions \ref{prop5} and \ref{prop8}, we
have the following.

\begin{prop}\label{prop9}
If $f:S^n\str S^n$ is an orientation preserving homeomorphism,
then there exists a homeomorphism $F:S^n\times I\str S^n\times I$
such that $F(x,0)=(x,0)$ and $F(x,1)=(f(x),1)$.
\end{prop}

\subsection{The equational system
$\mathcal{K}$}\label{appendixK}

To define the arrows of $\mathbf{K}$, we need an equational
system, denoted by $\mathcal{K}$. We start with an inductive
definition of \emph{terms}.
\begin{itemize}
\item[1.] For $n,m\in\omega$, the words $\mj_n:n\str n$,
$\tau_{n,m}:n+m\str m+n$, \\ $\md:2\str 1$, $\ed:0\str 1$,
$\ms:1\str 2$, $\es:1\str 0$, are \emph{primitive} terms.

\item[2.] If $f:n\str m$ and $g:m\str p$ are terms, then $(g\circ
f):n\str p$ is a term.

\item[3.] If $f_1:n_1\str m_1$ and $f_2:n_2\str m_2$ are terms,
then \\ $(f_1\otm f_2):n_1+n_2\str m_1+m_2$ is a term.

\item[4.] Nothing else is a term.
\end{itemize}

A \emph{type} is a word of the form $n\str m$, where
$n,m\in\omega$. We say that $n\str m$ is a type of a term $f:n\str
m$, and we say that this term has $n$ as the \emph{source} and $m$
as the \emph{target}. Usually, we omit the type in writing a term
and by a term we mean just its part before the colon symbol. Also,
we omit the outermost parenthesis in terms.

The \emph{language} of $\mathcal{K}$ consists of words of the form
$f=g$, where $f$ and $g$ are terms with the same type. Besides
$f=f$, the \emph{axiom schemata} of $\mathcal{K}$ are the
following.
\begin{equation}
  \tag{\emph{str}}
  f\otimes \mj_0=f=\mj_0\otimes f,\quad (f_1\otimes f_2)\otimes f_3=f_1\otimes(f_2\otimes f_3).\label{st}
\end{equation}
For $f:n\str m$, $g:m\str p$ and $h:p\str q$,

\vspace{-.5ex}
\begin{equation}
  \tag{\emph{cat}}
  f\circ \mj_n=f=\mj_m\circ f,\quad (h\circ g)\circ f=h\circ(g\circ f).\label{ct}
\end{equation}

\begin{equation}
  \tag{\emph{fun}}
  \mj_n\otimes\mj_m=\mj_{n+m},\quad (g_1\circ f_1)\otimes(g_2\circ f_2)=(g_1\otimes g_2)\circ (f_1\otimes f_2).
  \label{fn}
\end{equation}

\begin{equation}
  \tag{\emph{nat}}
  \tau_{m_1,m_2}\circ(f_1\otimes f_2)=(f_2\otimes f_1)\circ\tau_{n_1,n_2}.
  \label{nt}
\end{equation}

\begin{equation}
  \tag{\emph{inv}}
  \tau_{m,n}\circ\tau_{n,m}=\mj_{n+m}.
  \label{iv}
\end{equation}

\begin{equation}
  \tag{\emph{hex}}
  \tau_{n+m,p}=(\tau_{n,p}\otimes \mj_m)\circ(\mj_n\otimes\tau_{m,p}).
  \label{hx}
\end{equation}

\begin{equation}
  \tag{\emph{assoc}}
  \md\circ(\md\otimes\mj_1)=\md\circ(\mj_1\otimes\md).
  \label{as}
\end{equation}

\begin{equation}
  \tag{\emph{unit}}
  \md\circ(\ed\otimes\mj_1)=\mj_1=\md\circ(\mj_1\otimes\ed).
  \label{un}
\end{equation}

\begin{equation}
  \tag{\emph{coass}}
  (\ms\otimes\mj_1)\circ\ms=(\mj_1\otimes\ms)\circ\ms.
  \label{ca}
\end{equation}

\begin{equation}
  \tag{\emph{counit}}
  (\es\otimes\mj_1)\circ\ms=\mj_1=(\mj_1\otimes\es)\circ\ms.
  \label{cu}
\end{equation}

\begin{equation}
  \tag{\emph{Frob}}
  (\md\otimes\mj_1)\circ(\mj_1\otimes\ms)=\ms\circ\md=(\mj_1\otimes\md)\circ(\ms\otimes\mj_1).
  \label{fb}
\end{equation}

\begin{equation}
  \tag{\emph{com}}
  \md\circ\tau_{1,1}=\md.
  \label{cm}
\end{equation}

\begin{equation}
  \tag{\emph{cocom}}
  \tau_{1,1}\circ\ms=\ms.
  \label{cocm}
\end{equation}

\vspace{2ex}

The \emph{inference figures} of $\mathcal{K}$ are the following.
\[
\f{f=g}{g=f}\quad\quad \f{f=g\quad g=h}{f=h}
\]

\[
\f{f_1:n\str m=f_2:n\str m\quad\quad g_1:m\str p=g_2:m\str
p}{g_1\circ f_1=g_2\circ f_2}
\]

\[
\f{f_1=g_1\quad\quad f_2=g_2}{f_1\otimes f_2=g_1\otimes g_2}
\]

We say that terms $f$ and $g$ are \emph{equal}, when $f=g$ is a
theorem of $\mathcal{K}$, and we denote this by $f\equiv g$. The
relation $\equiv$ is an equivalence relation and $[f]$ is the
equivalence class of a term~$f$.


\begin{thebibliography}{99}

\bibitem{A96} {\sc L.\ Abrams}, {\it Two-dimensional topological quantum field theories
and Frobenius algebras}, \textbf{\textit{Journal of Knot Theory
and its Ramifications}}, vol.\ 5 (1996) pp.\ 569-587

\bibitem{B37} {\sc R.\ Brauer}, {\it On algebras which are connected with semisimple continuous groups},
\textbf{\textit{Annals of Mathematics}}, vol.\ 38 (1937) pp.\
857-872

\bibitem{B67} {\sc W.\ Browder}, {\it Diffeomorphisms of 1-connected manifolds}, \textbf{\textit{Transactions of the
American Mathematical Society}}, vol.\ 128 (1967) pp.\ 155-163

\bibitem{BG64} {\sc M.\ Brown} and {\sc H.\ Gluck}, {\it Stable structures on
manifolds. II. Stable manifolds}, \textbf{\textit{Annals of
Mathematics}}, Second Series, vol.\ 79 (1964) pp.\ 18-44

\bibitem{C70} {\sc J.\ Cerf}, {\it The pseudo-isotopy theorem for simply connected
differentiable manifolds}, \textbf{\textit{Manifolds-Amsterdam
1970}} (N.\ Kuiper, editor), Lecture Notes in Mathematics, vol.\
197, Springer, Berlin, 1970, pp.\ 76-82

\bibitem{D89} {\sc R.H.\ Dijkgraaf}, \textbf{\textit{A Geometric Approach To Two-Dimensional Conformal Field Theory}}, PhD thesis, University of Utrecht, 1989


\bibitem{D72} {\sc A.\ Dold}, \textbf{\textit{Lectures in Algebraic Topology}}, Springer, Berlin, 1972

\bibitem{DP03} {\sc K. Do\v sen and Z. Petri\' c}, {\it A Brauerian representation of split preorders},
\textbf{\textit{Mathematical Logic Quarterly}}, vol.\ 49, (2003),
pp.\ 579-586

\bibitem{DP003} --------, {\it Generality of proofs and its Brauerian representation},
\textbf{\textit{The Journal of Symbolic Logic}}, vol.\ 68, (2003),
pp.\ 740-750

\bibitem{DP06} --------, {\it Symmetric self-adjunctions:
A justification of Brauer's representation of Brauer's algebras},
\textbf{\textit{Proceedings of the Conference ``Contemporary
Geometry and Related Topics''}} (N.\ Bokan et al., editors),
Faculty of Mathematics, Belgrade, 2006, pp. 177-187

\bibitem{DP12} --------,
{\it Symmetric self-adjunctions and matrices},
\textbf{\textit{Algebra Colloquium}}, vol.\ 19, No. spec01 (2012)
pp.\ 1051-1082

\bibitem{KM63} {\sc M.\ Kervaire} and {\sc J.\ Milnor}, {\it Groups of homotopy spheres. I}, \textbf{\textit{Annals of
Mathematics}}, Second Series, vol.\ 77 (1963) pp.\ 504-537

\bibitem{K69} {\sc R.C.\ Kirby}, {\it Stable homeomorphisms and the annulus
conjecture}, \textbf{\textit{Annals of Mathematics}}, Second
Series, vol.\ 89 (1969) pp.\ 575-582

\bibitem{K03} {\sc J.\ Kock}, \textbf{\textit{Frobenius Algebras and 2D
Topological Quantum Field Theories}}, Cambridge University Press,
Cambridge, 2003

\bibitem{K13} {\sc M.\ Kreck}, {\it Orientation of manifolds}, \textbf{\textit{Manifold Atlas}},
\url{http://www.map.mpim-bonn.mpg.de/Orientation_of_manifolds},
2013

\bibitem{L62} {\sc W.B.R.\ Lickorish}, {\it A representation of orientable combinatorial 3-manifolds},
\textbf{\textit{Annals of Mathematics}}, Second Series, vol.\ 76
(1962) pp.\ 531-540

\bibitem{ML63} {\sc S.\ Mac Lane}, {\it Natural associativity and
commutativity}, \textbf{\textit{Rice University Studies, Papers in
Mathematics}}, vol.\ 49 (1963) pp.\ 28-46

\bibitem{ML65} --------, {\it Categorical algebra},
\textbf{\textit{Bulletin of the American Mathematical Society}},
vol.\ 71 (1965) pp.\ 40-106

\bibitem{ML71} --------, \textbf{\textit{Categories for the Working
Mathematician}}, Springer, Berlin, 1971 (expanded second edition,
1998)

\bibitem{M52} {\sc E.E.\ Moise}, {\it Affine structures in 3-manifolds. V. The
triangulation theorem and Hauptvermutung}, \textbf{\textit{Annals
of Mathematics}}, Second Series, vol.\ 56 (1952) pp.\ 96-114

\bibitem{PT17} {\sc Z.\ Petri\' c and S.\ Telebakovi\' c}, {\it A faithful 2-dimensional TQFT},
preprint, available at ArXiv (2017)

\bibitem{Q82} {\sc F.\ Quinn}, {\it Ends of maps. III. Dimensions 4 and 5},
\textbf{\textit{Journal of Differential Geometry}}, vol.\ 17
(1982) pp.\ 503-521

\bibitem{Q95} --------, {\it Lectures on axiomatic topological quantum field
theory}, \textbf{\textit{Geometry and Quantum Field Theory}}
(D.S.\ Freed and K.K.\ Uhlenbeck, editors), American Mathematical
Society, Providence, 1995, pp.\ 323-453

\bibitem{R24} {\sc T.\ Rad\' o} , {\it \" Uber den Begriff der Riemannschen Fl\" ache},
\textbf{\textit{Acta Universitatis Szegediensis}}, vol.\ 2 (1924)
pp.\ 101-121

\bibitem{S95} {\sc S.\ Sawin} , {\it Direct sum decompositions and indecomposable TQFTs},
\textbf{\textit{Journal of Mathematical Physics}}, vol.\ 36 (1995)
pp.\ 6673-6680

\bibitem{ST80} {\sc H.\ Seifert and W.\ Threlfall}, \textbf{\textit{A textbook of
topology}}, Pure and Applied Mathematics, 89, Academic Press, New
York, 1980 (English translation of 1934 classic German textbook)

\bibitem{T08} {\sc L.W.\ Tu}, \textbf{\textit{An Introduction to Manifolds}}, Springer, New York,
2008

\end{thebibliography}
\end{document}